\documentclass[a4paper, 13pt]{amsart}

\usepackage{amsmath,amssymb, amscd, color, hyperref}

\theoremstyle{plain}
\newtheorem{theorem}{\bf Theorem}[section]

\newtheorem{lemma}[theorem]{\bf Lemma}
\newtheorem{corollary}[theorem]{\bf Corollary}

\newtheorem{conjecture}[theorem]{Conjecture}

\theoremstyle{definition}
\newtheorem{example}[theorem]{\bf Example}

\newtheorem{definition}[theorem]{\bf Definition}

\newtheorem{remarks}[theorem]{\bf Remarks}

\newcommand{\N}{\mathbb N}
\newcommand{\Z}{\mathbb Z}
\newcommand{\R}{\mathbb R}

\newcommand{\red}{{\text{\rm red}}}

 \DeclareMathOperator{\ord}{ord}

 \DeclareMathOperator{\supp}{supp}

\newcommand{\ito}{\overset\sim\to}

\newcommand{\DP}{\negthinspace : \negthinspace}
\newcommand{\LK}{[\![}
\newcommand{\RK}{]\!]}

\renewcommand{\P}{\mathbb P}
\renewcommand{\t}{\, | \,}
\renewcommand{\time}{\negthinspace \times \negthinspace}

\numberwithin{equation}{section}

\address{University of Graz, NAWI Graz \\
Institute for Mathematics and Scientific Computing \\
Heinrichstra{\ss}e 36\\
8010 Graz, Austria}

\email{alfred.geroldinger@uni-graz.at, qinghai.zhong@uni-graz.at}
\urladdr{http://imsc.uni-graz.at/geroldinger, http://qinghai-zhong.weebly.com/}

\author{Alfred Geroldinger   and Qinghai Zhong}

\thanks{This work was supported by
the Austrian Science Fund FWF, Project Number P28864-N35}



\begin{document}
\title{Long sets of lengths with maximal elasticity}

\begin{abstract}
We introduce a new invariant describing the structure of sets of lengths in atomic monoids and domains. For an atomic monoid $H$, let $\Delta_{\rho} (H)$ be the set of all positive integers $d$ which occur as differences of arbitrarily long arithmetical progressions contained in sets of lengths having maximal elasticity $\rho (H)$. We study $\Delta_{\rho} (H)$ for transfer Krull monoids of finite type (including commutative Krull domains with finite class group) with methods from additive combinatorics, and also for a class of weakly Krull domains (including orders in algebraic number fields) for which we use ideal theoretic methods.
\end{abstract}

\maketitle

\medskip
\section{Introduction} \label{1}
\medskip

Let $H$ be a monoid or domain such that every (non-zero and non-unit) element can be written as a finite product of atoms. If $a=u_1 \cdot \ldots \cdot u_k$ is a factorization into atoms $u_1, \ldots, u_k$, then $k$ is called the length of this factorization and the set $\mathsf L (a) \subset \N$ of all possible factorization lengths is called the set of lengths of $a$. The system $\mathcal L (H) = \{ \mathsf L (a) \mid a \in H \}$ of all sets of lengths is a well-studied means of describing the non-uniqueness of factorizations of $H$. If there is some $a \in H$ such that $|\mathsf L (a)|>1$, then $\mathsf L (a^n) \supset \mathsf L (a) + \ldots + \mathsf L (a)$ whence $\mathsf L (a^n)$ has more than $n$ elements for every $n \in \N$. Weak ideal theoretic conditions on $H$ guarantee that all sets of lengths are finite. Then, apart from the trivial case where all sets of lengths are singletons, $\mathcal L (H)$ is a family of finite subsets of the integers containing arbitrarily long sets. Only in a couple of very special cases the system $\mathcal L (H)$ can be written down explicitly. In general, $\mathcal L (H)$ is described by parameters such as the set of distances $\Delta (H)$, the  elasticity $\rho (H)$, and others. We recall the definition of the elasticity $\rho (H)$. If $L \in \mathcal L (H)$, then $\rho (L)= \sup (L)/\min L$ is the elasticity of $L$ (thus $\rho (L)=1$ if and only if $|L|=1$). The elasticity $\rho (H)$ of $H$ is the supremum of all $\rho (L)$ over all $L \in \mathcal L (H)$, and we say that it is accepted if there is some $L \in \mathcal L (H)$ such that $\rho (H) = \rho (L)< \infty$.

The goal of the present paper is to study the possible differences of arbitrarily long arithmetical progressions contained in sets of lengths  having maximal possible elasticity.
More precisely, suppose that $H$ has accepted elasticity with $1 < \rho (H) < \infty$. Then let $\Delta_{\rho} (H)$ denote the set of all $d \in \N$ with the following property: for every $k \in \N$, there is some $L_k \in \mathcal L (H)$ with $\rho (L_k) = \rho (H)$ and $L_k = y_k + \big( L_k' \cup \{0, d, \ldots, \ell_k d\} \cup L_k''\big) \subset y_k + d \Z$, where $y_k \in \Z$, $\max L_k' < 0$, $\min L_k'' > \ell_k d$, and $\ell_k \ge k$. We study $\Delta_{\rho} (H)$ for transfer Krull monoids of finite type and for classes of weakly Krull monoids.

A transfer Krull monoid  of finite type is a monoid having a weak transfer homomorphism  to a monoid  of zero-sum sequences over a finite subset   of an abelian group. Transfer homomorphisms preserve factorization lengths which implies that the systems of sets of lengths of the two monoids coincide.
This setting includes commutative Krull  domains with finite class group, but also classes of not necessarily integrally closed noetherian domains, and classes of  non-commutative Dedekind prime rings (for a detailed discussion see the beginning of Section \ref{3}).

Let $H$ be a transfer Krull monoid over a finite abelian group $G$ such that $|G|\ge 3$. Then $\mathcal L (H) = \mathcal L \big( \mathcal B (G) \big) =: \mathcal L (G)$, whence sets of lengths of $H$ can be studied in the monoid $\mathcal B (G)$ of zero-sum sequences over $G$ and  methods from additive combinatorics can be applied.
This setting has found wide interest in the literature  (\cite{C-F-G-O16, Ge-HK06a, Sc16a}). Our main results on $\Delta_{\rho} (\cdot)$ for transfer Krull monoids are summarized after Conjecture \ref{3.20}. In a discussion preceding Lemma \ref{3.2} we  review the tools from zero-sum theory required for studying $\Delta_{\rho} (\cdot)$ and their state of the art. A central question in all studies of systems of sets of lengths is the so-called Characterization Problem, which asks whether for two non-isomorphic finite abelian groups $G$ and $G'$ (with Davenport constant $\mathsf D (G)\ge 4$)  the  systems of sets of lengths $\mathcal L (G)$ and $\mathcal L (G')$ can coincide. The standing conjecture is that this is not possible (see \cite[Section 6]{Ge16c} for a survey, and \cite{Ge-Sc16a, Ge-Zh17b, Zh17a} for recent progress), and the new invariant $\Delta_{\rho} (\cdot)$ turns out to be  a further useful tool in these investigations (Corollary \ref{3.19}).

Within factorization theory the case of (transfer) Krull monoids and domains is by far the best understood case. Much less is known in the non-Krull case. The most investigated class are Mori domains $R$ with non-zero conductor $\mathfrak f$, finite $v$-class group, and a finiteness condition on the factor ring $R/\mathfrak f$ (see \cite{Fo-Ho-Lu13a,Ka16b}). However, in the overwhelming number of situations only abstract arithmetical finiteness results are known but no precise results (such as in the Krull case). Mori domains, which are weakly Krull, have a defining family of one-dimensional local Mori domains which provides a strategy for obtaining precise results.  In Section \ref{4} we study $\Delta_{\rho} (\cdot)$ for such weakly Krull Mori domains and
 for their monoids of $v$-invertible $v$-ideals under natural algebraic finiteness assumptions which are satisfied, among others,  by orders in algebraic number fields (Theorem \ref{4.4}). This is done by studying  the local case first and then the local results are glued together with the help of the associated $T$-block monoid. Our results on $\Delta_{\rho} (\cdot)$ allow to reveal further classes of weakly Krull monoids which are not transfer Krull (Corollary \ref{4.6}).

\medskip
\section{Background on  sets of lengths} \label{2}
\medskip

For integers $a$ and $b$, we denote by $[a,b]=\{ x \in \Z \mid a \le x \le b\}$ the discrete interval between $a$ and $b$. Let $L \subset \Z$ be a subset. If $d \in \mathbb N$ \ and \ $\ell, M \in \mathbb N_0$, then $L$ \ is called an {\it almost arithmetical
progression} ({\rm AAP} for short) with  difference $d$, length $\ell$, and
bound $M$ if
\begin{equation} \label{eq:defAAP}
L = y + (L' \cup \{0, d , \ldots, \ell d\} \cup L'') \subset y + d \mathbb Z
\end{equation}
where $y\in \Z$,   $L' \subset [-M,-1]$, and $L'' \subset \ell d + [1, M]$.
If $L' \subset \Z$, then $L+L' = \{a+b \mid a \in L, b \in L'\}$ denotes the sumset. If $L = \{m_1, \ldots, m_k \} \subset \Z$ is finite with $k \in \N_0$ and $m_1 < \ldots < m_k$, then $\Delta (L)=\{m_i - m_{i-1} \mid i \in [2,k] \} \subset \N$ denotes the set of distances of $L$. If $L \subset \N$ is a subset of the positive integers, then  $\rho (L) = \sup L/ \min L$ denotes its elasticity, and for convenience we set $\rho ( \{0\})=1$.

Let $G$ be a finite abelian group. Let $r \in \N$ and    $(e_1, \ldots, e_r)$ be an $r$-tuple of elements of $G$. Then $(e_1, \ldots, e_r)$ is said to be independent if $e_i \ne 0$ for all $i \in [1,r]$ and if for all $(m_1, \ldots, m_r) \in \Z^r$ an equation $m_1e_1+ \ldots + m_re_r=0$ implies that $m_ie_i=0$ for all $i \in [1,r]$.   Furthermore, $(e_1, \ldots, e_r)$ is said to be a basis of $G$ if it is independent and $G = \langle e_1 \rangle \oplus \ldots \oplus \langle e_r \rangle$. For every $n \in \N$, we denote by $C_n$ an additive cyclic group of order $n$.

By a {\it monoid}, we mean an associative semigroup with unit element, and if not stated otherwise we use multiplicative notation. Let $H$ be a monoid with unit-element $1=1_H \in H$. We denote by $H^{\times}$ the group of invertible elements and say that $H$ is reduced if $H^{\times} = \{1\}$. Let $S \subset H$ be a subset and $a \in S$. Then $[S] \subset H$ denotes the submonoid generated by $S$, and $[a] = [\{a\}] = \{a^k \mid k \in \N_0\}$ is the submonoid generated by $a$.  We say that the subset  $S$ is  divisor-closed if $a, b \in H$ and $ab \in S$ implies that $a, b \in S$. We denote by $\LK S \RK$  the smallest divisor-closed submonoid containing $S$, and $\LK a \RK = \LK \{a\} \RK$ is the smallest divisor-closed submonoid of $H$ containing $a$. The monoid $H$ is said to be unit-cancellative if for each two elements $a,u \in H$  any of the  equations $au=a$ or $ua=a$ implies that $u \in H^{\times}$. Clearly, every cancellative monoid is unit-cancellative.

Suppose that $H$ is unit-cancellative. An element $u \in H$ is said to be irreducible (or an atom) if $u \notin H^{\times}$ and any equation of the form $u=ab$, with $a, b \in H$, implies that $a \in H^{\times}$ or $b \in H^{\times}$. Let  $\mathcal A (H)$ denote the set of atoms, and we say that $H$ is atomic if every non-unit is a finite product of atoms. If $H$ satisfies the ascending chain condition on principal left ideals and on principal right ideals, then $H$ is atomic  (\cite[Theorem 2.6]{Fa-Tr18a}). If $a \in H \setminus H^{\times}$ and $a=u_1 \cdot \ldots \cdot u_k$, where $k \in \N$ and $u_1, \ldots, u_k \in \mathcal A (H)$, then $k$ is a factorization length of $a$, and
\[
\mathsf L_H (a) = \mathsf L (a) = \{k \mid k \ \text{is a factorization length of} \ a \} \subset \N
\]
denotes the {\it set of lengths} of $a$. It is convenient to set $\mathsf L (a) = \{0\}$ for all $a \in H^{\times}$ (note that every divisor of an invertible element is again invertible). The family
\[
\mathcal L (H) = \{\mathsf L (a) \mid a \in H \}
\]
is called the {\it system of sets of lengths} of $H$, and
\[
\rho (H) = \sup \{\rho (L) \mid L \in \mathcal L (H) \} \in \R_{\ge 1} \cup \{\infty\}
\]
denotes the {\it elasticity} of $H$. We say that a monoid $H$  has {\it accepted elasticity}
\begin{itemize}
\item if it is atomic unit-cancellative with elasticity $\rho (H)< \infty$, and there is an $L \in \mathcal L (H)$ such that $\rho (L)=\rho (H)$.
\end{itemize}
Let  $H$  be a monoid with  accepted elasticity. Then $\sup L < \infty$ for every $L \in \mathcal L (H)$ and for a subset $S \subset H$,
\[
\Delta_H (S) =  \bigcup_{a \in S} \Delta ( \mathsf L_H (a) ) \ \subset \N
\]
denotes the set of distances of $S$.   Let $S \subset H$ be a divisor-closed submonoid and $a \in S$. Then $S^{\times} = H^{\times}$, $\mathcal A (S) = \mathcal A (H)$, $\mathsf L_S (a) = \mathsf L_H (a)$, and $\mathcal L (S) \subset \mathcal L (H)$. Furthermore, we have $\Delta_S (S) = \Delta_H (S)$ and we set $\Delta (S) = \Delta_S (S)$ and $\Delta (H) = \Delta_H (H)$. By definition we have $\Delta (H)= \emptyset$ if and only if $\rho (H)=1$.

For any set $P$, we denote by $\mathcal F (P)$ the free abelian monoid with basis $P$. If
\[
a = \prod_{p \in P}p^{\mathsf v_p (a)} \in \mathcal F (P), \quad \text{where} \quad \mathsf v_p \colon \mathcal F (P) \to \N_0 \quad \text{is the $p$-adic exponent} \,,
\]
then $|a|=\sum_{p \in P} \mathsf v_p (a) \in \N_0$ is the length of $a$.
Let $D$ be a monoid. A submonoid $H \subset D$ is said to be {\it saturated} if $a \in D$, $b \in H$, and ($ab \in H$ or $ba \in H$) imply that $a \in H$. A commutative monoid $H$ is Krull if its associated reduced monoid is a saturated submonoid of a free abelian monoid (\cite[Theorem 2.4.8]{Ge-HK06a}). A commutative domain  is Krull if and only if its monoid  of non-zero elements is a Krull monoid. The theory of commutative Krull monoids and domains is presented in \cite{HK98, Ge-HK06a}.

Let $G$ be an additive abelian group and $G_0 \subset G$ a nonempty subset. An element
\[
S = g_1 \cdot \ldots \cdot g_{\ell} = \prod_{g \in G_0} g^{\mathsf v_g (S)} \in \mathcal F (G_0)
\]
is said to be a zero-sum sequence if its sum $\sigma (S) = g_1 + \ldots + g_{\ell} = \sum_{g \in G_0} \mathsf v_g (S)g$ equals zero. Then the set $\mathcal B (G_0)$ of all zero-sum sequences over $G_0$ is a submonoid, and since $\mathcal B (G_0) \subset \mathcal F (G_0)$ is  saturated,  it is a commutative Krull monoid. If $S$ is  as above, then $|S|=\ell \in \N_0$ is the length of $S$ and $\supp (S) = \{g_1, \ldots, g_{\ell}\} \subset G$ denotes its support.
The monoid $\mathcal B (G_0)$  plays a crucial role in Section \ref{3}. It is usual to set $\mathcal L (G_0) :=
\mathcal L \bigl( \mathcal B(G_0) \bigr)$, $\mathcal A (G_0) := \mathcal A \bigl( \mathcal B (G_0) \bigr)$, $\rho (G_0) := \rho \bigl( \mathcal B (G_0) \bigr)$, and $\Delta (G_0) := \Delta \bigl( \mathcal B (G_0) \bigr)$ (although this is an abuse of notation, it will never lead to confusion). If $G_0$ is finite, then $\mathcal A (G_0)$ is finite and
\[
\mathsf D (G_0) = \max \{ |U| \mid U \in \mathcal A (G_0) \} \in \N
\]
denotes the {\it Davenport constant} of $G_0$.

\smallskip
Now we introduce the new arithmetical invariant, $\Delta_{\rho} (\cdot)$, to be studied in the present paper. For convenience we repeat the definition of the well-studied invariant $\Delta_1 (\cdot)$ (\cite[Definition 4.3.12]{Ge-HK06a}).

\smallskip
\begin{definition} \label{2.1}
Let $H$ be an atomic unit-cancellative  monoid.
\begin{enumerate}
\item Let $\Delta_{1} (H)$ denote the set of all $d \in \N$ having the following property{\rm \,:}
      \begin{itemize}
      \item[] For every $k \in \N$, there is some $L_k \in \mathcal L (H)$ which is an AAP with difference $d$ and  length at least $k$.
      \end{itemize}

\item Let $\Delta_{\rho} (H)$ denote the set of all $d \in \N$ having the following property{\rm \,:}
      \begin{itemize}
      \item[] For every $k \in \N$, there is some $L_k \in \mathcal L (H)$ which is an AAP with difference $d$, length at least $k$, and with $\rho (L_k)=\rho (H)$.
      \end{itemize}

\item We set $\Delta_{\rho}^* (H) =       \{ \min \Delta_H ( [ a ] ) \mid a \in H \ \text{with} \ \rho ( \mathsf L (a) ) = \rho (H) \}$.
\end{enumerate}
\end{definition}

By definition, we have
\begin{equation} \label{eq:basic1}
\Delta_{\rho} (H) \subset \Delta_1 (H) \subset \Delta (H) \,,
\end{equation}
and  $\Delta_{\rho}(H)=\emptyset$ if  $H$ does not have accepted elasticity.

\smallskip
The set $\Delta_1 (H)$ is studied with the help of the set $\Delta^* (H)$ which is defined as the set of all $d \in \N$ having the following property (\cite[Definition 4.3.12]{Ge-HK06a}){\rm \,:}
      \begin{itemize}
      \item[] There is a divisor-closed submonoid $S \subset H$ with $\Delta (S) \ne \emptyset$ and $d = \min \Delta (S)$.
      \end{itemize}
If $H$ is a commutative cancellative BF-monoid, then, by \cite[Proposition 4.3.14]{Ge-HK06a},
\begin{equation} \label{eq:Delta^*}
\Delta^* (H) = \{ \min \Delta ( \LK a \RK ) \mid a \in H \ \text{with} \ \Delta ( \LK a \RK ) \ne \emptyset \} \,.
\end{equation}
The sets $\Delta^* (H)$, called the set of minimal distances of $H$, and   $\Delta_1 (H)$ have found wide attention, so far mainly for transfer Krull monoids over finite abelian groups (\cite{Ge-Zh16a, Ge-Sc16a, Ge-Zh17b, Zh17a, Pl-Sc18a}).

\smallskip
In the present paper we study $\Delta_{\rho} (H)$, and the set $\Delta_{\rho}^* (H)$ is a technical tool to do so.
The relationship between the two sets is the topic of   Lemma \ref{2.4}. In particular, we have $\emptyset \ne \Delta_{\rho}^* (H)  \subset \Delta_{\rho} (H)$ (provided that $H$ has accepted elasticity $\rho (H)> 1$). Equations \eqref{eq:Delta^*} and \eqref{eq:Delta-rho^*} reveal the formal correspondence between $\Delta^* (H)$ and $\Delta_{\rho}^* (H)$ in the case of commutative monoids.
However, there  exist commutative monoids $H$ and divisor-closed submonoids $S \subset H$ with $\rho (S)=\rho (H)>1$ such that $\min \Delta (S) \notin \Delta_{\rho} (H)$ (use Theorem \ref{3.5} with $S=H \in \{ \mathcal B (C_4), \mathcal B (C_6), \mathcal B (C_{10} \}$). Thus,  in contrast to \eqref{eq:Delta^*}, in Equation \eqref{eq:Delta-rho^*} we cannot replace $\LK a \RK$ by an arbitrary divisor-closed submonoid.

In contrast to the formal similarity in the definitions, the invariants $\Delta_{\rho} (H)$ and $\Delta_1 (H)$ show a very different behavior (in particular for transfer Krull monoids over finite abelian groups, see Section \ref{3}). Thus the additional requirement on the elasticity is a very strong one.

\smallskip
We start with a technical lemma analyzing the set $\Delta_{\rho}^* (H)$.

\medskip
\begin{lemma} \label{2.2}
Let $S \subset H$ be a submonoid with $\Delta_H (S) \ne \emptyset$.
\begin{enumerate}
\item $\min \Delta_H (S) = \gcd \Delta_H (S)$.

\smallskip
\item If $H$ is commutative, then $\min \Delta ( \LK S \RK) = \min \Delta_H ( \LK S \RK) = \min \Delta_H (S)$ whence
      \begin{equation} \label{eq:Delta-rho^*}
      \Delta_{\rho}^* (H) =       \{ \min \Delta ( \LK  a \RK ) \mid a \in H \ \text{with} \ \rho ( \mathsf L (a) ) = \rho (H) \} \,.
      \end{equation}

\smallskip
\item Let $a, b  \in S$ with $\rho (\mathsf L_H (a))= \rho (\mathsf L_H (b))= \rho (H)$.  Then $\rho (\mathsf L_H (ab))= \rho (H)$. In particular, $\rho (\mathsf L_H (a^k))=\rho (H)$ for every $k \in \N$ and $\rho ( \LK a \RK ) = \rho (H)$.
\end{enumerate}
\end{lemma}

\begin{proof}
1. It is sufficient to prove that
$\min \Delta_H(S)\t d'$ for every $d'\in \Delta_H(S)$. Let $d=\min \Delta_H(S)$ and assume to the contrary that there exists $d'\in \Delta_H(S)$ such that $d\nmid d'$.

We set $d_0=\gcd(d,d')$. Then  $d_0<d$ and there exist $x, y\in \N$ such that $d_0=xd-yd'$. Let $a_1,a_2\in S$ such that $\{\ell_1, \ell_1+d\}\subset \mathsf L_H (a_1)$ and $\{\ell_2-d', \ell_2\}\subset \mathsf L_H (a_2)$. Thus $\{x\ell_1, x\ell_1+d, \ldots, x\ell_1+xd\}\subset \mathsf L_H (a_1^{x})$ and $\{y\ell_2-yd', y\ell_2-(y-1)d', \ldots, y\ell_2\}\subset \mathsf L_H (a_2^{y})$. Therefore $\{x\ell_1+y\ell_2, x\ell_1+y\ell_2+xd-yd'\}\subset \mathsf L_H (a_1^{x} a_2^{y})$ which implies that $d\le xd-yd'=d_0$, a contradiction.

\smallskip
2. Suppose that $H$ is commutative. Since $S \subset \LK S \RK$ and $\LK S \RK \subset H$ is divisor-closed, it follows that
\[
\min \Delta ( \LK S \RK) = \min \Delta_H ( \LK S \RK) \le  \min \Delta_H (S) \,.
\]
To verify the reverse inequality, let $b \in \LK S \RK$ with $\min \Delta ( \mathsf L_H (b) ) = \min \Delta ( \LK S \RK )$. There is a $c \in H$ such that $bc \in S$. Since $\mathsf L_H (b) + \mathsf L_H (c) \subset \mathsf L_H (bc)$, we infer that
\[
\min \Delta_H (S) \le \min \Delta (\mathsf L_H (bc)) \le \min \Delta ( \mathsf L_H (b)) = \min \Delta ( \LK S \RK) \,.
\]
In particular, if $S = [a]$, then $\min \Delta ( \LK a \RK) =  \min \Delta_H ( [a] )$ and hence the equation for $\Delta_{\rho}^* (H)$ follows.

\smallskip
3. Since $\mathsf L (a) + \mathsf L (b) \subset \mathsf L (ab)$, it follows that
\[
\min \mathsf L (ab) \le  \min \mathsf L (a) + \min \mathsf L (b ) \le  \max \mathsf L (a) +  \max \mathsf L (b) \le \max \mathsf L (ab) \,,
\]
and hence
\[
\rho (H) \ge \rho ( \mathsf L (a b ) ) = \frac{\max \mathsf L (ab)}{\min \mathsf L(ab)} \ge \frac{\max \mathsf L(a)+ \max \mathsf L (b)}{\min \mathsf L (a) + \min \mathsf L (b )} \ge \min \Big\{ \frac{\max \mathsf L(a)}{\min \mathsf L (a)}, \frac{\max \mathsf L(b)}{\min \mathsf L (b)} \Big\}  =  \rho (H) \,.
\]
The in particular statement follows by induction on $k$.
\end{proof}

\smallskip
We continue with a simple observation on  the structure of the sets $L_k$ -- popping up in the definition of $\Delta_{\rho} (H)$ -- for all monoids $H$ under consideration. To do so, we need a further definition. Let $d \in \N$, \ $M \in \N_0$, \ and \ $\{0,d\} \subset \mathcal D
\subset [0,d]$. A subset $L \subset \Z$ is called an {\it almost
arithmetical multiprogression} \ ({\rm AAMP} \ for
      short) \ with \ {\it difference} \ $d$, \ {\it period} \ $\mathcal D$,
      \  and \ {\it bound} \ $M$, \ if
\[
L = y + (L' \cup L^* \cup L'') \, \subset \, y + \mathcal D + d \Z
\]
where \ $y \in \mathbb Z$ is a shift parameter,
\begin{itemize}
\item  $L^*$ is finite nonempty with $\min L^* = 0$ and $L^* =
       (\mathcal D + d \Z) \cap [0, \max L^*]$, and

\item  $L' \subset [-M, -1]$ \ and \ $L'' \subset \max L^* + [1,M]$.
\end{itemize}
The following characterization of $\Delta_{\rho} (H)$ follows from the very definitions.

\medskip
\begin{lemma} \label{2.3}
Let $H$ be a monoid with accepted elasticity and with finite non-empty set of distances, and let $M \in \N$. Suppose that every $L \in \mathcal L (H)$ is an \text{\rm AAMP} with some difference $d \in \Delta (H)$ and bound $M$. Then $\Delta_{\rho} (H)$ is the set of all $d \in \N$ with the following property{\rm \,:} for every $k \in \N$ there is some $a_k \in H$ such that $\rho (\mathsf L (a_k)) = \rho (H)$ and
\[
\mathsf L (a_k) = y + (L' \cup \{0, d , \ldots, \ell d\} \cup L'') \subset y + d \mathbb Z
\]
where $y \in \Z$, $\ell \ge k$, $L' \subset [-M,-1]$, and $L'' \subset \ell d + [1,M]$.
\end{lemma}

\smallskip
The assumption in Lemma \ref{2.3}, that all sets of lengths are AAMPs with global bounds, is a well-studied property in factorization theory. It holds true, among others, for transfer Krull monoids of finite type (studied in Section \ref{3}) and for weakly Krull monoids (as studied in Theorem \ref{4.4}). We refer to \cite[Chapter 4.7]{Ge-HK06a} for a survey on settings where sets of lengths are AAMPs and also to \cite{Ge-Ka10a}. Thus, under this assumption,  the above lemma shows that the sets $L_k$ (in Definition \ref{2.1}.2 of $\Delta_{\rho} (H)$) have globally bounded beginning and end parts $L'$ and $L''$, and the goal is to study the set of possible distances in the  middle part which can get arbitrarily long.

\medskip
\begin{lemma} \label{2.4}
Let $H$ be a monoid with accepted elasticity.
\begin{enumerate}
\item If $\rho (H)>1$, then $\emptyset \ne \Delta_{\rho}^* (H)  \subset \Delta_{\rho} (H)$ and $\min \Delta_{\rho}^* (H) = \min \Delta_{\rho} (H)$. In particular, if $\rho (H)>1$ and $|\Delta (H)|=1$, then $\Delta_{\rho}^* (H)  = \Delta_{\rho} (H) = \Delta (H)$.

\smallskip
\item If $S \subset H$ is a divisor-closed submonoid with $\rho (S)=\rho(H)$, then $\Delta_{\rho} (S) \subset \Delta_{\rho} (H)$.

\smallskip
\item If $H$ is commutative and cancellative with  finitely many atoms up to associates, then \\ $\Delta_{\rho} (H) \subset \{ d \in \N \mid d \ \text{divides some} \ d' \in \Delta_{\rho}^* (H) \}$. In particular, $\max \Delta_{\rho}(H)=\max \Delta_{\rho}^* (H)$.

\smallskip
\item $\Delta_{\rho}(H)= \emptyset$ if and only if $\Delta_1 (H)=\emptyset$ if and only if $\Delta (H)=\emptyset$ if and only if $\rho (H)=1$.
\end{enumerate}
\end{lemma}

\begin{proof}
1. Suppose that $\rho (H)>1$. Then, by definition, there is an $a \in H$ with $\rho (\mathsf L (a))=\rho (H)>1$ whence $\Delta_H ( [a])\ne \emptyset$ and thus $\Delta_{\rho}^* (H) \ne \emptyset$. To verify that $\Delta_{\rho}^* (H) \subset \Delta_{\rho} (H)$, we set $d = \min \Delta_H (  [a]  )$. Then there is an $\ell \in \N$ such that $d \in \Delta (\mathsf L(a^{\ell}))$ and thus for every $k \in \N$ the set $\mathsf L (a^{k \ell})$ contains an arithmetical progression with difference $d$ and length at least $k$. Since $\min \Delta_H ([a]) = \gcd \Delta_H ([a])$ by Lemma \ref{2.2}.2, $\mathsf L (a^{k \ell})$ is an AAP with difference $d$ and length at least $k$ for every $k \in \N$. By Lemma \ref{2.2}.3, we have $\rho (\mathsf L_H ( a^{k \ell})) = \rho (H)$ and thus $d \in \Delta_{\rho} (H)$.

Since $\Delta_{\rho}^* (H)  \subset \Delta_{\rho} (H)$, it follows that $\min \Delta_{\rho} (H) \le \min \Delta_{\rho}^* (H)$. To verify the reverse inequality, let $d\in \Delta_{\rho} (H)$ be given.  Then there is an $a\in H$ such that $\mathsf L(a)$ is an AAP with
difference $d$, length at least $1$, and $\rho(\mathsf L(a))=\rho(H)$.
Thus  $\min \Delta_H ( [a ] )\in \Delta_{\rho}^*(H)$ by definition,  and clearly we have $\min \Delta_H ( [a ] )\le \min \Delta(\mathsf L(a))=d$.

If $\rho (H)>1$ and $|\Delta (H)|=1$, then the inclusions given in \eqref{eq:basic1} imply that $\Delta_{\rho}^* (H)  = \Delta_{\rho} (H) = \Delta (H)$.

\smallskip
2. Suppose that $S \subset H$ is divisor-closed with $\rho (S)=\rho (H)$. Then for every $a \in S$, we have $\mathsf L_S (a) = \mathsf L_H (a)$, and hence $\mathcal L (S) \subset \mathcal L (H)$. If $d \in \Delta_{\rho} (S)$, then by definition, for every $k \in \N$, there is some $L_k \in \mathcal L (S) \subset \mathcal L (H)$ which is an AAP with difference $d$, length at least $k$, and with $\rho (L_k)=\rho (S)=\rho (H)$, and thus $d \in \Delta_{\rho} (H)$.

\smallskip
3.  Clearly, the in particular statement follows from the asserted inclusion and from the fact that $\Delta_{\rho}^* (H) \subset \Delta_{\rho} (H)$ as shown in 1. We use several times the fact that finitely generated commutative monoids are locally tame and have accepted elasticity (\cite[Theorem 3.1.4]{Ge-HK06a}).

Without restriction we may suppose that $H$ is reduced, and we set $\mathcal A (H) = \{u_1, \ldots, u_t\}$ with $t \in \N$. Let $d \in \Delta_{\rho} (H)$ be given. Then for every $k \in \N$, there is a $b_k \in H$ such that $\rho \bigl( \mathsf L (b_k) \bigr) = \rho (H)$ and $\mathsf L (b_k)$ is an AAP with difference $d$ and length $\ell_k \ge k$. Since $\mathcal A (H)$ is finite, there are a nonempty subset $A \subset \mathcal A (H)$, say $A = \{u_1, \ldots, u_s\}$ with $s \in [1,t]$, a constant $M_1 \in \N_0$, and a subsequence $(b_{m_k})_{k \ge 1}$ of $(b_k)_{k\ge 1}$, say $b_{m_k}=c_k$ for all $k \in \N$, such that, again for all $k \in \N$,
\[
c_k = \prod_{i=1}^t u_i^{m_{k,i}} \quad \text{where} \quad m_{k,i} \ge k \ \text{for} \ i \in [1,s] \quad \text{and} \quad m_{k,i} \le M_1 \ \text{for} \ i \in [s+1,t]
\]
By Theorem \cite[Theorem 4.3.6]{Ge-HK06a} (applied to the monoid $\LK u_1 \cdot \ldots \cdot u_s \RK$), $L_k = \mathsf L ( \prod_{i=1}^s u_i^{m_{k,i}})$ is an AAP with difference $d'=\min \Delta ( \LK u_1 \cdot \ldots \cdot u_s \RK)$ for every $k \in \N$. Since $H$ is locally tame, \cite[Proposition 4.3.4]{Ge-HK06a} implies that there is a constant $M_2 \in \N_0$ such that for every $k \in \N$
\[
\max \mathsf L (c_k) \le \max L_k +M_2 \quad \text{and} \quad \min \mathsf L (c_k) \ge \min L_k - M_2 \,.
\]
Since for every $k \in \N$ there is a $y_k \in \N$ such that $y_k + L_k \subset \mathsf L (c_k)$, we infer that $d$ divides $d'$. Being a divisor-closed submonoid of a finitely generated monoid, the monoid $\LK u_1 \cdot \ldots \cdot u_s \RK$ is finitely generated by \cite[Proposition 2.7.5]{Ge-HK06a}. Thus there is an $a \in \LK u_1 \cdot \ldots \cdot u_s \RK$ such that $\rho (\mathsf L (a)) = \rho (\LK u_1 \cdot \ldots \cdot u_s \RK)$. Since $d$ divides $d' = \min \Delta ( \LK u_1 \cdot \ldots \cdot u_s \RK)$ and $d'$ divides $\min \Delta ( \LK a \RK)$, it follows that $d$ divides $\min \Delta ( \LK a \RK)$.

Next we verify that $\rho (\LK u_1 \cdot \ldots \cdot u_s \RK) = \rho (H)$ from which it follows that $\min \Delta ( \LK a \RK) \in \Delta_{\rho}^* (H)$ by Lemma \ref{2.2}.2. For $k \in \N$, we have
\[
\rho (H) = \frac{\max \mathsf L (c_k)}{\min \mathsf L (c_k)} \le \frac{\max L_k+ M_2}{\min L_k - M_2} \quad \text{and} \quad \frac{\max L_k}{\min L_k} \le \rho (\LK u_1 \cdot \ldots \cdot u_s \RK) \le \rho (H) \,.
\]
If $k$ tends to infinity, then $({\max L_k+ M_2})/({\min L_k - M_2})$ tends to ${\max \mathsf L (c_k)}/{\min \mathsf L (c_k)}$ which implies that $\rho (\LK u_1 \cdot \ldots \cdot u_s \RK) = \rho (H)$.

\smallskip
4. This follows from 1. and from the basic relation given in (\ref{eq:basic1}).
\end{proof}

\medskip
\begin{lemma} \label{2.5}
Let $H$ be a monoid with accepted elasticity. Then for every nonempty subset $\Delta \subset \Delta_{\rho} (H)$ there is a $d \in \Delta_{\rho}(H)$ such that $d \le \gcd \Delta$.
\end{lemma}

\begin{proof}
Let $\Delta = \{d_1, \ldots, d_n\} \subset \Delta_{\rho}(H)$ be a nonempty subset. For every $i \in [1,n]$ and every $k \in \N$ there is an $a_{i,k} \in H$ such that $\mathsf L (a_{i,k})$ is an AAP with difference $d_i$, length at least $k$, and with $\rho ( \mathsf L (a_{i,k}) ) = \rho (H)$. By Lemma \ref{2.2}.3,   $\mathsf L (a_{1,k} \cdot \ldots \cdot a_{n,k})$ has elasticity $\rho (H)$ for all $k \in \N$, and thus $d = \min \Delta_H ( [a_{1,k} \cdot \ldots \cdot a_{n,k}]) \in \Delta_{\rho}^* (H) \subset \Delta_{\rho} (H)$. If $k$ is sufficiently large, then $\gcd (d_1, \ldots, d_n)$ occurs as a distance of the sumset $\mathsf L (a_{1,k}) + \ldots + \mathsf L (a_{n,k})$.
Since the sumset
\[
\mathsf L (a_{1,k}) + \ldots + \mathsf L (a_{n,k}) \subset \mathsf L (a_{1,k} \cdot \ldots \cdot a_{n,k})
\]
and $d = \gcd \Delta_H ( [a_{1,k} \cdot \ldots \cdot a_{n,k}])$ by Lemma \ref{2.2}.1, $d$ divides any distance of $\Delta \big( \mathsf L (a_{1,k} \cdot \ldots \cdot a_{n,k}) \big)$ whence it divides $\gcd (d_1, \ldots, d_n)$.
\end{proof}

\medskip
\begin{lemma} \label{2.6}
Let $H=H_1 \times \ldots \times H_n$ where $n \in \N$ and $H_1, \ldots, H_n$ are atomic unit-cancellative monoids.
\begin{enumerate}
\item  Then ${\rho} (H)  = \sup \{\rho (H_1), \ldots, \rho (H_n)\}$, and $H$ has accepted elasticity if and only if there is some $i \in [1,n]$ such that $H_i$ has accepted elasticity $\rho (H_i)=\rho (H)$.

\item Let $s \in [1,n]$ and suppose that  $H_i$ has accepted elasticity  $\rho (H_i) = \rho (H)$ for all $i \in [1,s]$, and that  $H_i$ either does not have accepted elasticity or    $\rho (H_i)< \rho (H)$ for all $i \in [s+1, n]$. We  set
    \[
    \begin{aligned}
    \Delta' & = \big\{ \gcd  \{d_i \mid i \in I\}  \mid d_i \in  \Delta_{\rho} (H_i) \ \text{for all} \ i \in I, \emptyset \ne I \subset [1,s] \big\} \quad \text{and} \\
    \Delta'' & = \big\{ \gcd  \{d_i \mid i \in I\}  \mid d_i \in  \Delta^*_{\rho} (H_i) \ \text{for all} \ i \in I, \emptyset \ne I \subset [1,s] \big\} \,.
    \end{aligned}
    \]
    Then $\Delta' \subset \Delta_{\rho} (H)$, $\Delta'' \subset \Delta^*_{\rho} (H)$, and if $|\Delta (H_i)|=1$ for all $i \in [1,s]$, then $\Delta' = \Delta'' = \Delta^*_{\rho} (H) = \Delta_{\rho} (H)$.
\end{enumerate}
\end{lemma}

\begin{proof}
1. The formula for $\rho (H)$ follows from \cite[Proposition 1.4.5]{Ge-HK06a}, where a proof is given for cancellative monoids but the proof of the general case runs along the same lines. The formula for $\rho (H)$ immediately implies the second assertion.

\smallskip
2.(i) First we show that $\Delta' \subset \Delta_{\rho} (H)$. Let $\emptyset \ne I \subset [1,s]$, say $I = [1,r]$, and choose $d_i \in \Delta_{\rho} (H_i)$ for every $i \in [1,r]$.  For each $i \in [1,r]$ and every $\ell \in \N$ there is an $a_{i, \ell} \in H_i$ such that $\mathsf L (a_{i, \ell})$ is an AAP with difference $d_i$, length at least $2 \ell$, and with $\rho (\mathsf L (a_{i,\ell})) = \rho (H)$. Then
$\rho \big( \mathsf L (a_{1,\ell} \cdot \ldots \cdot a_{r,\ell}) \big) = \rho (H)$ by Lemma \ref{2.2}.3. Thus, for all sufficiently large $\ell$,  the sumset $\mathsf L (a_{1,\ell}) + \ldots +  \mathsf L (a_{r,\ell}) = \mathsf L (a_{1,\ell} \cdot \ldots \cdot a_{r,\ell})$ is an AAP with difference $\gcd (d_1, \ldots, d_r)$ and length at least $\ell$.

\smallskip
2.(ii) Second we show that $\Delta'' \subset \Delta^*_{\rho} (H)$. Let $\emptyset \ne I \subset [1,s]$, say $I = [1,r]$, and choose $d_i \in \Delta^*_{\rho} (H_i)$ for every $i \in [1,r]$. Thus there are $a_i \in H_i$ such that $\rho (\mathsf L (a_i))=\rho (H)$ and $\min \Delta_{H_i} ( [ a_i ]) = \min \Delta_H ( [ a_i ]) = d_i$ for all $i \in [1,r]$. Therefore, again for all $i \in [1,r]$, there is an $\ell_i \in \N$ such that $d_i \in \Delta ( \mathsf L (a_i^{\ell_i}))$ and thus, for every $k \in \N$, $\mathsf L (a_i^{2k \ell_i})$ contains an arithmetical progression with difference $d_i$ and length at least $2k$. Setting $\ell = \max (\ell_1, \ldots, \ell_r)$ we infer that
\[
\mathsf L ( (a_1 \cdot \ldots \cdot a_r)^{2k \ell}) = \mathsf L (a_1^{2k\ell}) + \ldots + \mathsf L (a_r^{2k\ell})
\]
is an AAP with difference $\gcd ( d_1, \ldots, d_r)$ and length at least $k$ for all sufficiently large $k$. Thus $\min \Delta_H ( [a_1 \cdot \ldots \cdot a_r]) = \gcd ( d_1, \ldots, d_r)$. Since
 $\rho (\mathsf L (a_1 \cdot \ldots \cdot a_r))=\rho (H)$ by Lemma \ref{2.2}.3,  it follows that $\gcd ( d_1, \ldots, d_r) = \min \Delta_H ( [a_1 \cdot \ldots \cdot a_r]) \in \Delta_{\rho}^* (H)$.

\smallskip
2.(iii)
Now suppose that $\Delta (H_i) = \{d_i\}$ for all $i \in [1,s]$. Then $\Delta_{\rho}^* (H_i)  = \Delta_{\rho} (H_i) = \Delta (H_i)$ by Lemma \ref{2.4}.1 and hence $\Delta' = \Delta''$. By 2.(i) and 2.(ii) it remains to show that $\Delta_{\rho} (H) \subset \Delta'$. Then all four sets are equal as asserted.

Let $d \in \Delta_{\rho} (H)$ and let $k \in \N$ be sufficiently large. Then there are $a_{1,k} \in H_1, \ldots, a_{s,k} \in H_s$ such that $\mathsf L (a_{1,k} \cdot \ldots \cdot a_{s,k})$ is an AAP with difference $d$, elasticity $\rho (H)$, and length at least $k$. Since $\Delta (H_i) = \{d_i\}$ for all $i \in [1,s]$,
\[
\mathsf L (a_{1,k} \cdot \ldots \cdot a_{s,k}) = \mathsf L (a_{1,k}) +   \ldots + \mathsf L ( a_{s,k})
\]
is a sumset of arithmetical progressions with differences $d_1, \ldots, d_s$.
After renumbering if necessary there is an $r \in [1,s]$ such that $|\mathsf L(a_{i,k})|>1$ for all  $i \in [1,r]$ and $|\mathsf L (a_{i,k})|=1$ for all $i \in [r+1,s]$. Thus we clearly obtain that $d \ge \gcd (d_1, \ldots, d_r)$.
Since $\mathsf L (a_{1,k} \cdot \ldots \cdot a_{s,k})$ is an AAP with difference $d$ and length at least $k$ with $k$ being sufficiently large, it follows that $\mathsf L (a_{1,k} \cdot \ldots \cdot a_{s,k}) \subset y + d \Z$ for some $y \in \Z$ (see \eqref{eq:defAAP}), which implies  that $d \mid d_i$ for all $i \in [1,r]$. Thus
$d = \gcd (d_1, \ldots, d_r)$ and hence $d \in \Delta'$.
\end{proof}

\medskip
\section{Transfer Krull monoids} \label{3}
\medskip

An atomic unit-cancellative monoid $H$ is said to be a {\it transfer Krull monoid} if one of the following two equivalent properties is satisfied:
\begin{itemize}
\item[(a)] There is a commutative Krull monoid $B$ and a weak transfer homomorphism $\theta \colon H \to B$.

\item[(b)] There is an abelian group $G$, a subset $G_0 \subset G$, and a weak transfer homomorphism $\theta \colon H \to \mathcal B (G_0)$.
\end{itemize}
In case (b) we say that $H$ is a transfer Krull monoid over $G_0$, and if $G_0$ is finite, then $H$ is said to be a transfer Krull monoid of finite type.
We do not repeat the technical definition of  weak transfer homomorphisms (introduced by Baeth and Smertnig in \cite{Ba-Sm15}) because we use only that they preserve sets of lengths. Therefore
$\mathcal L (H) = \mathcal L (G_0) $ (\cite[Lemma 4.2]{Ge16c}) which, by definition, implies that
\begin{equation} \label{eq:basic5}
\Delta (H) = \Delta (G_0), \ \Delta_{\rho} (H) = \Delta_{\rho} (G_0), \ \rho (H) = \rho  (G_0) \,,
\end{equation}
and $H$ has accepted elasticity if and only if $\mathcal B (G_0)$ has accepted elasticity. Note that, as  with other invariants, we use the abbreviations
\[
\Delta_1 (G_0) := \Delta_1 \bigl( \mathcal B (G_0) \bigr), \quad \Delta_{\rho}^* (G_0) := \Delta_{\rho}^* \bigl( \mathcal B (G_0) \bigr), \quad \text{and} \quad \Delta_{\rho} (G_0) := \Delta_{\rho} \bigl( \mathcal B (G_0) \bigr) \,.
\]
Every commutative Krull monoid (and thus every commutative Krull domain)  with class group $G$ is a transfer Krull monoid over the subset $G_0 \subset G$ containing prime divisors. In particular, if the class group $G$ is finite and every class contains a prime divisor (which holds true for holomorphy rings in global fields), then it is a transfer Krull monoid over $G$. Deep results, due to D. Smertnig, reveal large classes of bounded HNP (hereditary noetherian prime) rings to be transfer Krull (\cite{Sm13a, Ba-Sm15, Sm17a}). To mention one of these results in detail, let $\mathcal O$ be a ring of integers of an algebraic number field $K$, $A$ a central simple algebra over $K$, and $R$ a classical maximal $\mathcal O$-order of $A$. Then the monoid of cancellative elements of $R$ is transfer Krull if and only if every stably free left $R$-ideal is free, and if this holds, then it is a tranfer Krull monoid over a finite abelian group (namely a ray class group of $\mathcal O$).
We refer to \cite{Ge16c} for a detailed discussion of commutative Krull monoids with finite class group and of further transfer Krull monoids.

\smallskip
Let $H$ be a transfer Krull monoid over a finite abelian group $G$. The system $\mathcal L (H) = \mathcal L (G)$, together with all parameters controlling it, is a central object of interest in factorization theory (see  \cite{Sc16a} for a survey). By \eqref{eq:basic1} and Lemma \ref{2.4}.1, we have
\[
\Delta^*_{\rho} (G) \subset \Delta_{\rho} (G) \subset \Delta_1 (G) \subset \Delta (G)  \,.
\]
The set $\Delta (G)$ is an interval by \cite{Ge-Yu12b}, but $\Delta_1 (G)$ is far from being an interval (\cite{Pl-Sc18a}). A characterization when $\Delta_1 (G)$ is an interval can be found in \cite{Zh17a}. We have $\max \Delta_1 (G) = \max \{\mathsf r (G)-1, \exp (G)-2 \}$ (for $|G|\ge 3$, by \cite{Ge-Zh16a}). This section will reveal that $\Delta_{\rho} (G)$ is quite different from $\Delta_1 (G)$.

\smallskip
We start with a  result for transfer Krull monoids over arbitrary finite subsets. It shows that in finitely generated commutative Krull monoids $H$ with finite class group (and without restriction on the classes containing prime divisors)  a large variety of finite sets can be realized as $\Delta_{\rho} (H)$ sets (Lemma \ref{2.5} shows that not every finite set can be realized as a $\Delta_{\rho} (\cdot)$ set of some monoid;  see also Lemmas \ref{2.6} and \ref{4.3}). In contrast to this we will see that the set $\Delta_{\rho} (H)$ is extremely restricted if the set of classes containing prime divisors is very large.

\medskip
\begin{theorem} \label{3.1}~

\begin{enumerate}
\item Let $H$ be a transfer Krull monoid over a finite subset $G_0$. Then $H$ has accepted elasticity $\rho (H)=\rho (G_0) \le \mathsf D (G_0)/2$ and equality holds if $G_0=-G_0$.

\smallskip
\item For every finite set  $\Delta = \{d_1, \ldots, d_n\} \subset \N$  there exists a finitely generated commutative Krull monoid $H$ with finite class group such that $\{ \gcd  \{d_i \mid i \in I \} \mid \emptyset \ne I \subset [1,n] \} = \Delta^*_{\rho} (H) =   \Delta_{\rho} (H)$.

\smallskip
\item  If $H$ be a transfer Krull monoid over a subset $G_0$ of a finite abelian group $G$ with  $\rho (H) = \mathsf D (G)/2$,  then $\langle G_0 \rangle =G$ and $\Delta_{\rho} (H) \subset \Delta_{\rho} (G)$.
\end{enumerate}
\end{theorem}

\begin{proof}
1. By \eqref{eq:basic5}, we have $\mathcal L (H) = \mathcal L (G_0)$ and hence $\rho (H)=\rho (G_0)$.
Since the set $G_0$ is finite, the monoid $\mathcal B (G_0)$ is finitely generated whence the elasticity $\rho (G_0)$ is accepted (\cite[Theorems 3.1.4 and 3.4.2]{Ge-HK06a}). The statements on $\rho (G_0)$ follow from \cite[Theorem 3.4.11]{Ge-HK06a}.

\smallskip
2. Let $\Delta = \{d_1, \ldots, d_n\} \subset \N$ be a finite set. We start with the following assertion.

\smallskip
\begin{enumerate}
\item[{\bf A.}\,] For every $i \in [1,n]$, there is a finite abelian group $G_i$ and a subset $G_i' \subset G_i$ such that $\Delta_{\rho} (G_i') = \Delta (G_i')=\{d_i\}$ and $\rho (G_i') = 2$.

\end{enumerate}

\noindent
{\it Proof of} \,{\bf A}.\, We do the construction for a given $d \in \N$ and omit all indices. If $d = 1$, then $G = C_8 = \{0,g, \ldots ,7g\}$ and $G' = \{g, 3g\}$ have the required properties. Suppose that $d \ge 2$. Consider a finite abelian group $G$,  independent elements $e_1, \ldots, e_{d-1} \in G$ with $\ord (e_1) = \ldots = \ord (e_{d-1}) = 2d$, and set $e_0 = -(e_1+ \ldots + e_{d-1})$. It is easy to check that $G' = \{e_0, e_1, \ldots, e_{d-1}\}$ satisfies $\rho (G')=2$ and $\Delta (G')=\{d\}$ (for details of a more general construction see \cite[Proposition 4.1.2]{Ge-HK06a}).
\qed[Proof of {\bf A}]

We set
\[
G_0 = \biguplus_{i=1}^n G_i' \ \subset \ G = G_1 \oplus \ldots \oplus G_n \quad \text{and} \quad H = \mathcal B (G_0) \,.
\]
Then $H = \mathcal B (G_1') \times \ldots \times \mathcal B (G_n')$ is a finitely generated commutative Krull monoid with finite class group. By Lemma \ref{2.6}.1, $H$ has accepted elasticity $\rho (H)=2$ and
\[
\big\{ \gcd  \{d_i \mid I \subset [1,n] \}  \mid \emptyset \ne I \subset [1,n] \big\} =  \Delta^*_{\rho} (H) = \Delta_{\rho} (H) \,.
\]

\smallskip
3. Let $H$ be a transfer Krull monoid over $G_0$ such that $\rho (H)=\mathsf D (G)/2$. Then 1. shows that
\[
\mathsf D(G)/2 = \rho (H) \le \mathsf D (G_0)/2 \le \mathsf D (G)/2 \,.
\]
Thus
\[
\mathsf D (G) = \mathsf D (G_0) \le \mathsf D ( \langle G_0 \rangle ) \le \mathsf D (G) \,,
\]
and since proper subgroups of $G$ have a strictly smaller Davenport constant (\cite[Proposition 5.1.11]{Ge-HK06a}), it follows that $\langle G_0 \rangle = G$.

Since $\rho (H)=\rho (G_0)$ and $\rho (G)=\mathsf D (G)/2$ by 1., we obtain that
 $\rho (G_0) = \rho (G)$. Since $\Delta_{\rho} (H) = \Delta_{\rho} (G_0)$ and $\mathcal B (G_0) \subset \mathcal B (G)$ is a divisor-closed submonoid, the assertion follows from Lemma \ref{2.4}.2.
\end{proof}

\smallskip
Let all notation  be as in Theorem \ref{3.1}.3. Since $\Delta_{\rho} (H) \ne \emptyset$ and $\Delta_{\rho} (G)$ will turn out to be small (Conjecture \ref{3.20}), we have $\Delta_{\rho} (H) = \Delta_{\rho} (G)$ in many situations (as it holds true in the case $G_0=G$).

In the remainder of this section we study $\Delta_{\rho} (G)$ for  finite abelian groups $G$. Suppose that
\begin{equation} \label{eq:basic6}
G \cong C_{n_1} \oplus \ldots \oplus C_{n_r} \quad \text{and set} \quad \mathsf D^* (G) = 1 + \sum_{i=1}^r (n_i-1) \,,
\end{equation}
where $1 < n_1 \t \ldots \t n_r$,  $n_r=\exp (G)$ is the exponent of $G$, and $r=\mathsf r (G)$ is the rank of $G$. Thus $\mathsf r (G) = \max \{\mathsf r_p (G) \mid p \in \P\}$ is the maximum of all $p$-ranks $\mathsf r_p (G)$ over all primes $p \in \P$.

The next lemma, Lemma \ref{3.2}, reveals that the study of $\Delta_{\rho} (G)$ needs information on the Davenport constant $\mathsf D (G)$ as well as (at least some basic) information on the structure of minimal zero-sum sequences having length $\mathsf D (G)$. Although studied since the 1960s, the precise value of the Davenport constant is known only in a very limited number of cases. Clearly, we have
$\mathsf D^* (G) \le \mathsf D (G)$ and  since the 1960s it is known that equality holds if $\mathsf r (G) \le 2$ or if $G$ is a $p$-group. Further classes of groups have been found where equality holds and also where it does not hold, but a good understanding of this phenomenon is still missing. Even less is known on the inverse problem, namely on the structure of minimal zero-sum sequences having length $\mathsf D (G)$. The structure of such sequences is clear for cyclic groups and for elementary $2$-groups, and recently the structure was determined for rank two groups. For general groups, even  harmless looking questions (such as whether each minimal zero-sum sequence of length $\mathsf D (G)$ does contain an element of order $\exp (G)$) are open. In this section we study $\Delta_{\rho} (G)$ for all classes of groups where at least some information on the inverse problem is available.

Recall that $\Delta (G)=\emptyset$ if and only if $|G| \le 2$ whence we will always assume that $|G| \ge 3$.

\medskip
\begin{lemma} \label{3.2}
Let $G$ be a finite abelian group with $|G|\ge 3$.
\begin{enumerate}
\item  For $A \in \mathcal B (G)$  the following statements are equivalent{\rm \,:}
       \begin{enumerate}
       \item[(a)] $\rho \big( \mathsf L (A) \big) = \mathsf D (G)/2$.

       \item[(b)] There are $k, \ell \in \N$ and $U_1, \ldots, U_k, V_1, \ldots, V_{\ell} \in \mathcal A (G)$ with $|U_1|=\ldots=|U_k|=\mathsf D (G)$, $|V_1|=\ldots=|V_{\ell}|=2$ such that $A = U_1 \cdot \ldots \cdot U_k=V_1 \cdot \ldots \cdot V_{\ell}$.
       \end{enumerate}

\smallskip
\item  For a subset $G_0 \subset G$      the following statements are equivalent{\rm \,:}
       \begin{enumerate}
       \item[(a)] $G_0 = \supp (A)$ for some $A \in \mathcal B (G)$ with $\rho (\mathsf L (A))=\mathsf D (G)/2$.

       \item[(b)] $G_0=-G_0$ and for every $g \in G_0$ there is some $A \in \mathcal A (G_0)$ with $g \t A$ and $|A|=\mathsf D (G)$.
       \end{enumerate}

\smallskip
\item $\Delta_{\rho}^* (G) = \{\min \Delta (G_0) \mid G_0 = \supp (A) \ \text{\rm for some} \ A \in \mathcal B (G) \ \textrm{\rm with} \ \rho (\mathsf L (A))=\mathsf D (G)/2 \}$.
\end{enumerate}
\end{lemma}

\noindent{\it Comment on 1.} If $U_1, \ldots, U_m \in \mathcal A (G)$ with $|U_1|= \ldots = |U_m|=\mathsf D (G)$, then obviously we obtain an equation of the form $U_1(-U_1) \cdot \ldots \cdot U_m(-U_m) = V_1 \cdot \ldots \cdot V_{m \mathsf D (G)}$ with $|V_i|=2$ for all $i \in [1, m \mathsf D (G)]$. But there are also equations $U_1 \cdot \ldots \cdot U_k=V_1 \cdot \ldots \cdot V_{\ell}$ with all properties as in 1.(b) and with $k$  odd (\cite{Fa-Zh16a}).

\begin{proof}
1. (a) $\Rightarrow$ (b) We set $L = \mathsf L (A)$ and suppose that $\rho (L)=\mathsf D (G)/2$. If $A = 0^mC$, with $m \in \N_0$ and $C \in \mathcal B (G \setminus \{0\})$, then
\[
\frac{\mathsf D (G)}{2} = \frac{\max L}{\min L} = \frac{m+\max \mathsf L (C)}{m+\min \mathsf L (C)} \le \frac{\max \mathsf L (C)}{\min \mathsf L (C)} \le \frac{\mathsf D (G)}{2} \,,
\]
whence $m=0$. Suppose that
\[
U_1 \cdot \ldots \cdot U_k = A = V_1 \cdot \ldots \cdot V_{\ell} \,,
\]
with $k = \min \mathsf L (A)$, $\ell = \max \mathsf L (A)$, and $U_1, \ldots, U_k, V_1, \ldots, V_{\ell} \in \mathcal A (G)$. Then $\rho (L)=\ell/k=\mathsf D (G)/2$ and
\[
2 \ell \le \sum_{i=1}^{\ell} |V_i|=|A|= \sum_{i=1}^k |U_i| \le k \mathsf D (G) \,.
\]
This implies that $|A|=2 \ell = k \mathsf D (G)$, $|V_1|=\ldots=|V_{\ell}|=2$, and $|U_1|= \ldots=|U_k| = \mathsf D (G)$.

(b) $\Rightarrow$ (a) Suppose that $A = U_1 \cdot \ldots \cdot U_k=V_1 \cdot \ldots \cdot V_{\ell}$ where $U_1, \ldots, U_k, V_1, \ldots, V_{\ell}$ are as in (b). Then we infer that
\[
\min \mathsf L (A) \mathsf D (G) \le k \mathsf D (G) = |A| = 2 \ell \le 2 \max \mathsf L (A)
\]
and hence
\[
\frac{\mathsf D (G)}{2} \le \frac{\max \mathsf L (A)}{\min \mathsf L (A)} = \rho \big( \mathsf L (A) \big) \le \frac{\mathsf D (G)}{2} \,.
\]

\smallskip
2.(a) $\Rightarrow$ (b) This follows from 1.

(b) $\Rightarrow$ (a) We set $G_0 = \{g_1, -g_1, \ldots, g_k, -g_k\}$. For every $i \in [1,k]$, let $A_i \in \mathcal A (G_0)$ with $g_i \t A_i$ and $|A_i|=\mathsf D (G)$, and set $A = \prod_{i=1}^k (-A_i)A_i$. Then $\supp (A)=G_0$ and $\rho \bigl( \mathsf L (A) \bigr) = \mathsf D (G)/2$.

\smallskip
3. Since for every $A \in \mathcal B (G)$ we have $\LK A \RK = \mathcal B ( \supp (A) )$, the assertion follows from 2.
\end{proof}

\medskip
\begin{corollary} \label{3.3}
Let $G$ be a finite abelian group with $|G|\ge 3$.
\begin{enumerate}
\item $\Delta_{\rho}^* (G) \subset \Delta_{\rho} (G) \subset \{ d \in \N \mid d \ \text{divides some} \ d' \in \Delta_{\rho}^* (G) \}$.

\smallskip
\item $\max \Delta_{\rho} (G) = \max \Delta_{\rho}^* (G) = \max \{ \min \Delta (G_0) \mid G_0 = \supp \big(  (-U)U \big), U \in \mathcal A (G_0) \ \text{with} \ |U|= \mathsf D (G) \}$.
\end{enumerate}
\end{corollary}

\begin{proof}
1. Since $\mathcal B (G)$ is finitely generated, this follows from Lemma \ref{2.4}.

\smallskip
2. The first equality follows from 1. Then  Lemma \ref{3.2}.3 implies that
\[
\max \Delta_{\rho}^* (G)= \max\{\min \Delta (G_0) \mid G_0 = \supp (A) \ \text{\rm for some} \ A \in \mathcal B (G) \ \textrm{\rm with} \ \rho (\mathsf L (A))=\mathsf D (G)/2 \} \,.
\]
Let $A \in \mathcal B (G)$ with $G_0 = \supp (A)$ and $\rho (\mathsf L (A))=\mathsf D (G)/2$.  Then, by Lemma \ref{3.2}, $G_0=-G_0$ and $A = U_1 \cdot \ldots \cdot U_k$ with $U_1, \ldots, U_k \in \mathcal A (G)$ and $|U_1|=\ldots = |U_k|=\mathsf D (G)$. Then $G_1 = \supp \big( (-U_1)U_1 \big) \subset G_0$ and $\min \Delta (G_0) \le \min \Delta (G_1)$. Thus the assertion follows.
\end{proof}

\smallskip
Let $G$ be a finite abelian group and let $g \in G$ with \ $\ord(g) = n \ge 2$. For every
sequence $S = (n_1g) \cdot \ldots \cdot (n_{\ell}g) \in \mathcal F
(\langle g \rangle)$, \ where $\ell \in \N_0$ and $n_1, \ldots, n_{\ell}
\in [1, n]$, \ we define its $g$-norm
\[
\| S \|_g = \frac{n_1+ \ldots + n_{\ell}}n  \,.
\]
Note that, $\sigma (S) = 0$ implies that   $n_1 + \ldots + n_{\ell}
\equiv 0 \mod n$ whence $\| S \|_g \in \N_0$.

\medskip
\begin{lemma} \label{3.4}
Let $G$ be a finite abelian group with $|G|\ge 3$ and $G_0 \subset G$ be a subset.
\begin{enumerate}
\item  If $-G_0=G_0$, then $\min \Delta (G_0)$ divides $\gcd \, \{ |U|-2 \mid U \in \mathcal A (G_0) \}$.

\smallskip
\item If $r\ge 2$,  $(e_1,\ldots, e_r)$  independent,  $\ord(e_i)=n_i$ for all $i\in [1,r]$ where $n_1 \t \ldots \t n_r$, $n_r > 2$, $e_0=e_1+\ldots+e_r$, and $G_0=\{e_1,-e_1,\ldots, e_r,-e_r, e_0,-e_0\}$, then $\min \Delta(G_0)=1$.

\smallskip
\item If \ $\langle G_0 \rangle = \langle g \rangle$ for some $g \in G_0$ and $\Delta (G_0) \ne \emptyset$, then \ $\min \Delta (G_0) = \gcd  \bigl\{ \|V\|_g -1 \,\bigm|\, V \in \mathcal A (G_0) \bigr\}$.
\end{enumerate}
\end{lemma}

\begin{proof}
1. If $U=g_1 \cdot \ldots \cdot g_{\ell} \in \mathcal A (G_0)$, then $(-U)U = \prod_{i=1}^{\ell} \big( (-g_i)g_i \big)$ whence $\{2, \ell\} \subset \mathsf L \big( (-U)U \big)$ and so $\gcd \Delta (G_0)$ divides $\ell - 2$.

\smallskip
2.  Since $e_0=e_1+\ldots+e_r$, we have $\ord(e_0)=n_r>2$. We distinguish two cases. First, suppose that $n_1>2$. Then
\[
W=e_0^{n_r-1}e_1 \cdot \ldots \cdot  e_{r-1}(-e_r)^{n_r-1}\in \mathcal A(G_0)
\]
and  $W^2=e_0^{n_r}\cdot (-e_r)^{n_r}\cdot \big( e_0^{n_r-2}e_1^2 \cdot \ldots \cdot e_{r-1}^2(-e_r)^{n_r-2} \big)$ is a product of three atoms whence $\min \Delta(G_0)=1$.

Now, we suppose that $n_1=2$, and let $t \in [1,r-1]$ such that $n_1=\ldots=n_t=2$ and $n_{t+1}>2$. Then
\[
 S_1=e_0e_1 \cdot \ldots \cdot e_{t}(-e_{t+1}) \cdot \ldots \cdot (-e_r)\in \mathcal A(G_0) \quad \text{and} \quad
 S_2=e_0^{n_r-1}e_1 \cdot \ldots \cdot e_r\in \mathcal A(G_0)\,.
\]
Then
\[
S_1^2= \big( e_0^2(-e_{t+1})^2 \cdot \ldots \cdot (-e_r)^2\big)  e_1^2 \cdot \ldots \cdot e_t^2
\]
is a product of $t+1$ atoms and
\[
S_2^2=e_0^{n_r} \cdot \big(e_0^{n_r-2}e_{t+1}^2 \cdot \ldots \cdot e_{r}^2 \big) \cdot e_1^2\cdot \ldots \cdot e_t^2
\]
is a product of $t+2$ atoms. Thus $\min \Delta(G_0)\t \gcd(t+1-2, t+2-2)=1$ which implies   that $\min \Delta(G_0)=1$.

\smallskip
3. See \cite[Lemma 6.8.5]{Ge-HK06a}.
\end{proof}

\medskip
\begin{theorem}\label{3.5}
Let $H$ be a transfer Krull monoid over a finite abelian group $G$  with $|G|\ge 3$. Then $1\in \Delta_{\rho}(H)$ if and only if $G$ is not cyclic of order $4,6$ or $10$.
\end{theorem}

\begin{proof}
By \eqref{eq:basic5}, it is sufficient to prove the assertion for $\mathcal B (G)$ instead of $H$. We distinguish two cases.

\medskip
\noindent
CASE 1: \ $\mathsf r(G)\ge 2$.

By Corollary \ref{3.3}.1, it is sufficient to prove that $1\in \Delta_{\rho}^*(G)$.
For each prime $p$ dividing $|G|$, we denote by $G_p$ the Sylow $p$-subgroup of $G$. Since $\mathsf r(G)\ge 2$,   there exists a Sylow-$p$  subgroup $G_p$ such that $\mathsf r(G_p)\ge 2$. We distinguish two subcases.

\medskip
\noindent
CASE 1.1: \  There exists a  Sylow $p$-subgroup $G_p$ such that $\mathsf r(G_p)\ge 2$ and $\exp(G_p)\ge 3$.

Then there exists a subgroup $H$ of $G$ with $p\nmid |H|$ such that $G\cong G_p\oplus H$ (clearly, we may have  $H = \{0\}$). Let $A$ be an atom of $\mathcal B(G)$ with length $|A|=\mathsf D(G)$. Thus for every $g$ dividing $A$, there exists a unique pair $(f_g, h_g)$   with $f_g\in G_p$ and $h_g\in H$ such that $g=f_g+h_g$. Since $\langle\supp(A)\rangle=G$, there must exist $g \in \supp ( A)$ such that $\ord(f_g)=\exp(G_p)$. Therefore we can find $e_2,\ldots, e_{\mathsf r(G_p)}$ such that $G_p = \langle f_g\rangle \oplus\langle e_2\rangle \oplus \ldots \oplus \langle e_{\mathsf r(G_p)} \rangle $. There are group  isomorphisms
$$\phi: G\longrightarrow G \ \text{ by } \phi(f_g)=f_g+e_2,\  \phi(e_i)=e_i \text{ for each $i\in [2, \mathsf r(G_p)]$,  and }\phi(h)=h \text{ for each $h\in H$}\,, $$
and $$\psi: G\longrightarrow G \ \text{ by } \psi(f_g)=f_g-e_2,\  \psi(e_i)=e_i \text{ for each $i\in [2, \mathsf r(G_p)]$,  and }\psi(h)=h \text{ for each $h\in H$}\,. $$
It follows that $\phi(A)$ and $\psi(A)$ are atoms of length $\mathsf D(G)$. We consider the set
\[
G_0=\supp\big( (-A) A \, \phi((-A)A) \, \psi((-A)A) \big) \,.
\]
Obviously, we have
$G_0=-G_0$ and for every $a \in G_0$ there is some $A' \in \mathcal A (G_0)$ with $a \t A'$ and $|A'|=\mathsf D (G)$. Thus, by Lemma \ref{3.2},  it is sufficient to prove $\min \Delta(G_0)=1$. Since
\[
\{g, -g, \phi(g), \psi(g)\}=\{g, -g,  g+e_2, g-e_2\}\subset G_0
\]
 and
\[
 \mathsf L(g^{\ord(g)}(-g)^{\ord(g)})=\{2, \ord(g)\} \quad \text{and} \quad \mathsf L(g^{\ord(g)-2}(g+e_2)(g-e_2)(-g)^{\ord(g)})=\{2, \ord(g)-1\} \,,
\]
it follows that $\min \Delta(G_0)\t \gcd \{\ord(g)-2, \ord(g)-3\}$. Since  $\ord(g)\ge \exp(G_p)\ge 3$, we obtain that $\min \Delta(G_0)=1$.

\medskip
\noindent
CASE 1.2: \  There is no  Sylow $p$-subgroup $G_p$ such that $\mathsf r(G_p)\ge 2$ and $\exp(G_p)\ge 3$.

Let $G_p$ be the Sylow $p$-subgroup with $\mathsf r(G_p)\ge 2$. Then $p=2$, $G_2$ is an elementary $2$-group, and $G\cong C_2^{\mathsf r(G)}\oplus H$, where $H$ is a cyclic subgroup of odd order.

 Let $A$ be an atom of $\mathcal B(G)$ with length $|A|=\mathsf D(G)$. There  exists an element $g_0\in \supp ( A)$ such that $\ord(g_0)$ is even and hence $g_0=f_0+h_0$, where $f_0\in G_2\setminus\{0\}$ and $h_0\in H$.
  We can find $e_2,\ldots, e_{\mathsf r(G)}$ with $\ord(e_i)=2$ for each $i\in [2, \mathsf r(G)]$ such that $G_2\cong \langle f_0\rangle \oplus\langle e_2\rangle \oplus \ldots \oplus \langle e_{\mathsf r(G)} \rangle $. Then we can construct  two group  isomorphisms
$$\phi: G\longrightarrow G \ \text{ by } \phi(f_0)=e_2,\ \phi(e_2)=f_0,\   \phi(e_i)=e_i \text{ for each $i\in [3, \mathsf r(G)]$,  and }\phi(h)=h \text{ for each $h\in H$}\,, $$
and $$\psi: G\longrightarrow G \ \text{ by } \psi(f_0)=f_0+e_2,\  \psi(e_i)=e_i \text{ for each $i\in [2, \mathsf r(G)]$,  and }\psi(h)=h \text{ for each $h\in H$}\,. $$
It follows that $\phi(A)$ and $\psi(A)$ are atoms of length $\mathsf D(G)$. We consider the set
\[
G_0=\supp\big( (-A)A \, \phi((-A)A) \, \psi((-A)A) \big) \,.
\]
Obviously, we have
$G_0=-G_0$ and for every $a \in G_0$ there is some $A' \in \mathcal A (G_0)$ with $a \t A'$ and $|A'|=\mathsf D (G)$. Thus it is sufficient to prove $\min \Delta(G_0)=1$.

Note that $\{g_0, -g_0, \phi(g_0), \psi(g_0)\}=\{g_0, -g_0, e_2+h_0, g_0+e_2\}\subset G_0$.
If $\ord(g_0)=2$, then $h_0=0$ and  $\mathsf L(g_0^2e_2^2(g_0+e_2)^2)=\{2, 3\}$   imply that $\min \Delta(G_0)=1$.
Suppose that $\ord(g_0)\ge 4$. Since
\[
 \mathsf L(g_0^{\ord(g_0)}(-g_0)^{\ord(g_0)})=\{2, \ord(g_0)\} \quad \text{and} \quad \mathsf L(g_0^{\ord(g_0)-2}(g_0+e_2)^2(-g_0)^{\ord(g_0)})=\{2, \ord(g_0)-1\} \,,
\]
it follows that $\min \Delta(G_0)$ divides $\gcd \{\ord(g_0)-2, \ord(g_0)-3\}=1$.

\bigskip
\medskip
\noindent
CASE 2: \ $\mathsf r(G)=1$.

Let $|G|=n$ and $g \in G$ with $\ord (g)=n$. First, we suppose that $n$ is odd.
Then $g^n$ and $(2g)^n$ are atoms of length $\mathsf D(G)=n$, and we set $G_0=\{g, -g, 2g, -2g\}$. Then $G_0=-G_0$ and for every $h \in G_0$ there is some $A \in \mathcal A (G_0)$ with $h \t A$ and $|A|=\mathsf D (G)$. It is sufficient to prove that $\min \Delta(G_0)=1$. In fact, by Lemma \ref{3.4}.1, we obtain that $\min \Delta(G_0)$ divides $\gcd \{|g^n|-2, |g^{n-2}(2g)|-2\}=1$.

Now we suppose that $n$ is even and distinguish two subcases.

\medskip
\noindent
CASE 2.1: \  $n\notin\{4,6,10\}$.

It is sufficient  to show that  $1\in \Delta_{\rho}^*(G)$. We distinguish two cases.

First, suppose that there exists an odd positive divisor $m$ of $\frac{n}{2}+1$ such that $m\ge 5$. Then $\gcd(m,n)=1$. Let  $n=m(t+1)-2$, where $t\ge 1$. Then $A_1=(mg)^tg^{m-2}$, $A_2=(mg)g^{n-m}$, $A_3=(mg)^{2t+1}g^{m-4}$, and $A_4=g^n$ are atoms. Since $A_1^2A_2=A_3A_4$, we obtain that $1\in \Delta(\{g,-g, mg, -mg\})$. By the definition of $\Delta_{\rho}^*(G)$ and  Lemma \ref{3.2}, we have that  $1\in \Delta_{\rho}^*(G)\subset \Delta_{\rho}(G) $.

Second, suppose that for every odd positive divisor $m$ of  $\frac{n}{2}+1$, we have $m\le 3$. Then $\frac{n}{2}+1=2^{\alpha}$ or $\frac{n}{2}+1=3\cdot 2^{\alpha-1}$ where $\alpha\in \N$.  Thus $n+4 \in \{2(2^{\alpha}+1), 2(3\cdot 2^{\alpha-1}+1)\}$. Since $n\notin \{4,6,10\}$, we obtain that $\alpha\ge 3$. Let $g \in G$ with $\ord (g)=n$, and  $n+4=2k$, where $k$ is odd with $k\ge 9$. It follows  that $\gcd(k,n)=1$ and $A_5=(kg)g^{n-k}$,  $A_6=(kg)^3g^{2n-3k}$, $A_7=g^n$ are atoms. Since $A_5^3=A_6A_7$, we have that $1\in \Delta(\{g, -g, kg, -kg\})$.

\medskip
\noindent
CASE 2.2: \  $n\in\{4,6,10\}$.

We have to show that $1\not\in \Delta_{\rho}(G)$.
If $n\in \{4,6\}$, it is easy to check $\Delta_{\rho}(G)=\{n-2\}$. Suppose that $n=10$. Let $$G_0=\bigcup_{A\in \mathcal A(G) \text{ with }|A|=n}\supp(A)=\bigcup_{m\in [1,9] \text{ and }\gcd(m,10)=1}\{mg\}=\{g, -g, 3g, -3g\}\,.$$
Then Lemma \ref{3.2} implies that $\min \Delta(G_0)=\min \Delta_{\rho}^*(G)$.

By Lemma \ref{2.4}.1, we infer that   $\min \Delta_{\rho}^*(G)=\min \Delta_{\rho}(G)$. By Lemma \ref{3.4}.3, $\min \Delta(G_0)=\gcd \{\| V\|_g-1 \mid V\in \mathcal A(G_0) \} =2$ which implies that $1\notin \Delta_{\rho}(G)$.
\end{proof}

\medskip
\begin{lemma}  \label{3.6}
Let  $G = C_{m} \oplus C_{mn}$  with $n\geq 1$  and $m \ge 2$.  A sequence $S$
over $G$ of length $\mathsf D (G) = m+mn-1$ is a minimal zero-sum
sequence if and only if it has one of the following two forms{\rm
\,:}
\begin{itemize}
\medskip
\item \[
      S = e_1^{\ord (e_1)-1} \prod_{i=1}^{\ord (e_2)}
      (x_{i}e_1+e_2),
      \]where
      \begin{itemize}\item[(a)] $\{e_1, e_2\}$ is a basis of $G$,
      \item[(b)] $x_1, \ldots, x_{\ord (e_2)}  \in
      [0, \ord (e_1)-1]$ and $x_1 + \ldots + x_{\ord (e_2)} \equiv 1
      \mod \ord (e_1)$. \end{itemize} In this case, we say that $S$ is of type I(a) or I(b) according to whether $\ord(e_2)=m$ or $\ord(e_2)=mn>m$.

\medskip
\item \[
      S = f_1^{sm - 1} f_2^{(n-s)m+\epsilon}\prod_{i=1}^{m-\epsilon} ( -x_{i} f_1 +
      f_2),
      \] where
\begin{itemize}
      \item[(a)] $\{f_1, f_2\}$ is a generating set for  $G$ with $\ord (f_2) =
      mn$ and $\ord(f_1)>m$,
\item[(b)] $\epsilon\in [1,m-1]$  and
       $s \in [1, n-1]$,
        \item[(c)] $x_1, \ldots, x_{m-\epsilon} \in [1, m-1]$ with $x_1 + \ldots + x_{m-\epsilon} = m-1$,  \item[(d)] either  $s=1$ or
      $mf_1 = mf_2$, with both holding when $n=2$, and
      \item[(e)] either $\epsilon\geq 2$  or $mf_1\neq mf_2$.\end{itemize} In this case, we say that $S$ is of type II.
\end{itemize}
\end{lemma}

\begin{proof}
The characterization of minimal zero-sum sequences of maximal length over groups of rank two was done in a series of papers by   Gao, Geroldinger, Grynkiewicz, Reiher, and Schmid. For the formulation used above we refer to \cite[Main Proposition 7]{Ge-Gr-Yu15}.
\end{proof}

\medskip
\begin{theorem} \label{3.7}
Let $H$ be a transfer Krull monoid over a finite abelian group $G$. If $G$ has rank two, then  $\Delta_{\rho} (H) = \{1\}$.
\end{theorem}

\begin{proof}
By \eqref{eq:basic5}, we may consider $\mathcal B (G)$ instead of $H$.
Let $G = C_m\oplus C_{mn}$ with $n\in \N$, $m\ge 2$ and let $S$ be a minimal zero-sum sequence of length $\mathsf D(G)$ over $G$.  By Corollary \ref{3.3}.2, it suffices to prove that $1\in  \Delta \big( \supp((-S)S) \big)$.   We distinguish two cases depending on Lemma \ref{3.6}.

\medskip\noindent
CASE 1: \ $S = e_1^{\ord (e_1)-1} \prod_{i=1}^{\ord (e_2)}
      (x_{i}e_1+e_2)$ is of type $I$ in Lemma \ref{3.6}, where $(e_1,e_2)$ is a basis of $G$.
\smallskip

If $x_1=\ldots =x_{\ord(e_2)}$, then $\ord(e_2)x_1\equiv 1\pmod{\ord(e_1)}$ and hence $\gcd(\ord(e_1), \ord(e_2))=1$, a contradiction.
Suppose that $|\{x_1,\ldots, x_{\ord(e_2)}\}|\ge 2$. Then there exists a subsequence $Y=y_1\cdot \ldots \cdot y_{\ord(e_2)}$ of $X=x_1^2\cdot \ldots \cdot x_{\ord(e_2)}^2$  such that $\sigma(Y)\not\equiv 1\pmod{\ord(e_1)}$. Let $\sigma(Y)\equiv \ord(e_1)-a\pmod{\ord(e_1)}$, where $a\in [0,\ord(e_1)-2]$.  Then
\[
T_1=e_1^a\prod_{i=1}^{\ord(e_2)}(y_ie_1+e_2) \quad \text{ and } \quad T_2=e_1^{\ord(e_1)-2}\prod_{i=1}^{\ord (e_2)}
            (x_{i}e_1+e_2)^2T_1^{-1}
\]
are two minimal zero-sum sequences with $S^2=e_1^{\ord(e_1)}\cdot T_1\cdot T_2$ whence $1\in \Delta \big(\supp((-S)S) \big)$.

\medskip\noindent
 CASE 2: \ $S = f_1^{sm - 1} f_2^{(n-s)m+\epsilon}\prod_{i=1}^{m-\epsilon} ( -x_{i} f_1 +
      f_2)$ is of type $II$ in Lemma \ref{3.6}, where $(f_1', f_2)$ is a basis with $\ord(f_1')=m$, $\ord(f_2)=mn$ and $f_1=f_1'+\alpha f_2$, $\alpha\in [1,mn-1]$.
\smallskip

Since $sm-1+(n-s)m+\epsilon=nm+\epsilon-1\ge nm$, we have that $2\big( (n-s)m+\epsilon \big)\ge mn$ or $2(sm-1)\ge mn$. We distinguish two subcases.

\smallskip\noindent
CASE 2.1: \  $2((n-s)m+\epsilon)\ge mn$.
\smallskip

 Then $S^2=f_2^{nm}\cdot f_1^{2sm-2}f_2^{nm-2sm+2\epsilon}\prod_{i=1}^{m-\epsilon} ( -x_{i} f_1 +
      f_2)^2$. It suffices to prove that $$W=f_1^{2sm-2}f_2^{nm-2sm+2\epsilon}\prod_{i=1}^{m-\epsilon} ( -x_{i} f_1 + f_2)^2=(f_1'+\alpha f_2)^{2sm-2}f_2^{nm-2sm+2\epsilon}\prod_{i=1}^{2m-2\epsilon} ( -y_{i} f_1' +(1-\alpha y_i) f_2)\,,$$ where $y_1\cdot \ldots \cdot y_{2m-2\epsilon}=x_1^2\cdot\ldots \cdot x_{m-\epsilon}^2$,
       is a product of two atoms, since this implies that $1\in \Delta \big(\supp( (-S)S ) \big)$.
Note that $\sum_{i\in[1, 2m-2\epsilon]}y_i=2(\sum_{i\in [1, m-\epsilon]}x_i)=2m-2$ and  $|W|=mn+2m-2>\mathsf D(G)$, whence $W$ is not an atom.

Suppose that $s=1$. Then
\[
W=(f_1'+\alpha f_2)^{2m-2}\cdot \prod_{i=1}^{2m-2\epsilon}(-y_if_1'+(1-\alpha y_i)f_2)\cdot f_2^{nm-2m+2\epsilon} \,.
\]
Let $T$ be an atom dividing $W$, say
\[
T=(f_1'+\alpha f_2)^{r}\cdot \prod_{i\in I}(-y_if_1'+(1-\alpha y_i)f_2)\cdot f_2^{r'} \,, \quad \text{where}
\]
\[
I\subset [1, 2m-2\epsilon], \  r\equiv \sum_{i\in I}y_i \pmod{m}, \quad \text{ and} \quad  \alpha (r-\sum_{i\in I}y_i)+|I|+r'\equiv 0\pmod {nm} \,.
\]
If $r= \sum_{i\in I}y_i$, then $|I|+r'\ge mn$ which implies that $I=[1, 2m-2\epsilon]$ and $r'=nm-2m+2\epsilon$. Therefore $WT^{-1}\t (f_1'+\alpha f_2)^{2m-2-r}$,  a contradiction to $\ord(f_1)=\ord(f_1'+\alpha f_2)>m$. Thus $ |r- \sum_{i\in I}y_i|=m$.

Now we assume to the contrary that there exist three atoms $T_1, T_2$, and $T_3$ such that $T_1T_2T_3\t W$, say
\begin{align*}
     T_1&=(f_1'+\alpha f_2)^{r_1}\cdot \prod_{i\in I_1}(-y_if_1'+(1-\alpha y_i)f_2)\cdot f_2^{r_1'}\,,\\
     T_2&=(f_1'+\alpha f_2)^{r_2}\cdot \prod_{i\in I_2}(-y_if_1'+(1-\alpha y_i)f_2)\cdot f_2^{r_2'}\,,\\
     T_3&=(f_1'+\alpha f_2)^{r_3}\cdot \prod_{i\in I_3}(-y_if_1'+(1-\alpha y_i)f_2)\cdot f_2^{r_3'}\,.
\end{align*}
Then $ |r_1- \sum_{i\in I_1}y_i|=|r_2- \sum_{i\in I_2}y_i|=|r_3- \sum_{i\in I_3}y_i|=m$, a contradiction to $r_1+r_2+r_3\le 2m-2$ and $\sum_{i\in I_1}y_i+\sum_{i\in I_2}y_i+\sum_{i\in I_3}y_i\le 2m-2 $.

\smallskip
Suppose that $s\ge 2$. Then $mf_1=mf_2$ whence $\alpha m\equiv m \pmod{mn}$.
Let $T$ be an atom dividing $W$, say
\[
T=(f_1'+\alpha f_2)^{r}\cdot \prod_{i\in I}(-y_if_1'+(1-\alpha y_i)f_2)\cdot f_2^{r'} \,, \quad \text{where}
\]
\[
I\subset [1, 2m-2\epsilon], \ r\equiv \sum_{i\in I}y_i \pmod{m}, \quad \text{and} \quad  \alpha (r-\sum_{i\in I}y_i)+|I|+r'\equiv 0\pmod {nm} \,.
\]
If $r= \sum_{i\in I}y_i$, then $nm\le |I|+r'\le 2m-2\epsilon+nm-2sm+2\epsilon\le nm-2sm+2m$ which  implies that $s=1$, a contradiction.

We claim that $r-\sum_{i\in I}y_i\in \{(2s-1)m, -m\}$.
  If $r<\sum_{i\in I}y_i$, then $\sum_{i\in I}y_i-r=m$. We assume that $r>\sum_{i\in I}y_i$. Then $r-\sum_{i\in I}y_i\in \{m,\ldots, (2s-1)m\}$. Since $|I|+r'\le 2m-2\epsilon+nm-2sm+2\epsilon= nm-2sm+2m$ and $\alpha m\equiv m \pmod{mn}$, we have  $r-\sum_{i\in I}y_i\in \{(2s-2)m, (2s-1)m\}$. If $r-\sum_{i\in I}y_i=(2s-2)m$, then $|I|+r'= 2m-2\epsilon+nm-2sm+2\epsilon$ whence  $T=W$, a contradiction. Therefore $r-\sum_{i\in I}y_i\in \{(2s-1)m, -m\}$.

Now we assume to the contrary that there exist three atoms $T_1, T_2$, and $T_3$ such that $T_1T_2T_3\t W$, say
\begin{align*}
     T_1&=(f_1'+\alpha f_2)^{r_1}\cdot \prod_{i\in I_1}(-y_if_1'+(1-\alpha y_i)f_2)\cdot f_2^{r_1'}\,,\\
     T_2&=(f_1'+\alpha f_2)^{r_2}\cdot \prod_{i\in I_2}(-y_if_1'+(1-\alpha y_i)f_2)\cdot f_2^{r_2'}\,,\\
     T_3&=(f_1'+\alpha f_2)^{r_3}\cdot \prod_{i\in I_3}(-y_if_1'+(1-\alpha y_i)f_2)\cdot f_2^{r_3'}\,.
      \end{align*}
Then there exist two distinct $i,j\in [1,3]$, say $i=1,j=2$, such that  $ r_1- \sum_{i\in I_1}y_i=r_2- \sum_{i\in I_2}y_i=(2s-1)m$. Thus $2sm-2\ge r_1+r_2\ge 2(2s-1)m$, a contradiction.

\smallskip\noindent
CASE 2.2: \  $2(sm-1)\ge mn$.
\smallskip

Then $2s\ge n+1$. Therefore $mf_1=mf_2$ which implies that $\alpha m\equiv m \pmod{mn}$ and $\ord(f_1)=mn$.
Since $S^2=f_1^{nm}\cdot f_1^{2sm-nm-2}f_2^{2nm-2sm+2\epsilon}\prod_{i=1}^{m-\epsilon} ( -x_{i} f_1 +
      f_2)^2$,  it suffices to prove that
      \begin{align*}
     W&=f_1^{2sm-nm-2}f_2^{2nm-2sm+2\epsilon}\prod_{i=1}^{m-\epsilon} ( -x_{i} f_1 + f_2)^2\\
     &=(f_1'+\alpha f_2)^{2sm-nm-2}f_2^{2nm-2sm+2\epsilon}\prod_{i=1}^{2m-2\epsilon} ( -y_{i} f_1' +(1-\alpha y_i) f_2)\,,
      \end{align*} where $y_1\cdot\ldots\cdot y_{2m-2\epsilon}=x_1^2\cdot\ldots\cdot x_{m-\epsilon}^2$,
       is a product of two atoms since this implies that $1\in \Delta \big(\supp( (-S)S ) \big)$. Note that $\sum_{i\in[1, 2m-2\epsilon]}y_i=2(\sum_{i\in [1, m-\epsilon]}x_i)=2m-2$, $2sm-nm-2<mn$ and  $|W|=mn+2m-2>\mathsf D(G)$ whence $W$ is not an atom.

Let $T$ be an atom dividing $W$, say
\[
T=(f_1'+\alpha f_2)^{r}\cdot \prod_{i\in I}(-y_if_1'+(1-\alpha y_i)f_2)\cdot f_2^{r'} \,, \text{where}
\]
\[
I\subset [1, 2m-2\epsilon], \  r\equiv \sum_{i\in I}y_i \pmod{m} , \quad \text{ and} \quad \alpha (r-\sum_{i\in I}y_i)+|I|+r'\equiv 0\pmod {nm} \,.
\]
Suppose that $2s=n+1$. Then
\[
W=(f_1'+\alpha f_2)^{m-2}f_2^{(n-1)m+2\epsilon}\prod_{i=1}^{2m-2\epsilon} ( -y_{i} f_1' +(1-\alpha y_i) f_2) \,,
\]
and we assume to the contrary that there exist three atoms $T_1, T_2$, and $T_3$ such that $T_1T_2T_3\t W$, say
\begin{align*}
     T_1&=(f_1'+\alpha f_2)^{r_1}\cdot \prod_{i\in I_1}(-y_if_1'+(1-\alpha y_i)f_2)\cdot f_2^{r_1'}\,,\\
     T_2&=(f_1'+\alpha f_2)^{r_2}\cdot \prod_{i\in I_2}(-y_if_1'+(1-\alpha y_i)f_2)\cdot f_2^{r_2'}\,,\\
     T_3&=(f_1'+\alpha f_2)^{r_3}\cdot \prod_{i\in I_3}(-y_if_1'+(1-\alpha y_i)f_2)\cdot f_2^{r_3'}\,.
\end{align*}
Then there exist two distinct $i,j\in [1,3]$, say $i=1,j=2$, such that  $ r_1- \sum_{i\in I_1}y_i=r_2- \sum_{i\in I_2}y_i=0$. Thus $2nm\le |I_1|+r_1'+|I_2|+r_2'< (n-1)m+2\epsilon+2m-2\epsilon=nm+m$, a contradiction.

\smallskip
Suppose that $2s\ge n+2$. Consider the atom $T$.
 If $r=\sum_{i\in I}y_i$,  then $nm\le |I|+r'\le 2m-2\epsilon+2nm-2sm+2\epsilon\le (2n-2s+2)m\le nm$.  Therefore  $I=[1,2m-2\epsilon]$ and $r'=2nm-2sm+2\epsilon$ which infers that $T=W$, a contradiction.

 We claim that $r-\sum_{i\in I}y_i\in \{(2s-n-1)m, -m\}$.
 If $r<\sum_{i\in I}y_i$, then $\sum_{i\in I}y_i-r=m$. We assume that $r>\sum_{i\in I}y_i$. Then $r-\sum_{i\in I}y_i\in \{m,\ldots, (2s-n-1)m\}$. Since $|I|+r'\le 2m-2\epsilon+2nm-2sm+2\epsilon\le (2n-2s+2)m$ and $\alpha m\equiv m \pmod{mn}$, we have  $r-\sum_{i\in I}y_i\in \{(2s-n-2)m, (2s-n-1)m\}$. If $r-\sum_{i\in I}y_i=(2s-n-2)m$, then   $I=[1,2m-2\epsilon]$ and $r'=2nm-2sm+2\epsilon$ which infers that $T=W$, a contradiction. Therefore $r-\sum_{i\in I}y_i\in \{(2s-n-1)m, -m\}$.

Assume to the contrary that there exist three atoms $T_1, T_2$, and $T_3$ such that $T_1T_2T_3\t W$, say
\begin{align*}
     T_1&=(f_1'+\alpha f_2)^{r_1}\cdot \prod_{i\in I_1}(-y_if_1'+(1-\alpha y_i)f_2)\cdot f_2^{r_1'}\,,\\
     T_2&=(f_1'+\alpha f_2)^{r_2}\cdot \prod_{i\in I_2}(-y_if_1'+(1-\alpha y_i)f_2)\cdot f_2^{r_2'}\,,\\
     T_3&=(f_1'+\alpha f_2)^{r_3}\cdot \prod_{i\in I_3}(-y_if_1'+(1-\alpha y_i)f_2)\cdot f_2^{r_3'}\,.
      \end{align*}
Then there exist two distinct $i,j\in [1,3]$, say $i=1,j=2$, such that  $ r_1- \sum_{i\in I_1}y_i=r_2- \sum_{i\in I_2}y_i=(2s-n-1)m$. Thus $2sm-nm-2\ge r_1+r_2\ge 2(2s-n-1)m$ and hence $(n+2)m-2\ge 2sm\ge (n+2)m$, a contradiction.
\end{proof}

The characterization of all minimal zero-sum sequences over groups $C_2 \oplus C_2 \oplus C_{2n}$, as given in the next lemma, is due to Schmid (\cite[Theorem 3.13]{Sc11b}).

\medskip
\begin{lemma} \label{3.8}
 Let  $G = C_2\oplus C_2\oplus C_{2n}$ with $n \ge 2$. Then $A \in \mathcal F(G) $ is a minimal zero-
sum sequence of length $\mathsf D(G)$ if and only if there exists a basis $(f_1 ,f_2 ,f_3 )$ of $G$, where
$\ord(f_1 ) = \ord(f_2 ) = 2$ and $\ord(f_3 ) = 2n$, such that $A$ is equal to one of the following sequences{\rm \,:}
\begin{enumerate}
\item[(i)] $f_3^{v_3}(f_3+f_2)^{v_2}(f_3+f_1)^{v_1}(-f_3+f_2+f_1)$ with $v_1, v_2,v_3\in \N$ odd, $v_3\ge v_2\ge v_1$, and $v_3+v_2+v_1=2n+1$.

\smallskip

\item[(ii)] $f_3^{v_3}(f_3+f_2)^{v_2}(af_3+f_1)(-af_3+f_2+f_1)$ with $v_2,v_3\in \N$ odd, $v_3\ge v_2$, $v_2+v_3=2n$, and $a\in [2, n-1]$.

 \smallskip

\item[(iii)] $f_3^{2n-1}(af_3+f_2)(bf_3+f_1)(cf_3+f_2+f_1)$ with $a+b+c=2n+1$ where $a\le b\le c$ and $a,b\in [2,n-1]$, $c\in [2, 2n-3]\setminus\{n,n+1\}$.

\smallskip

\item[(iv)] $f_3^{2n-1-2v}(f_3+f_2)^{2v}f_2(af_3+f_1) \big((1-a)f_3+f_2+f_1 \big)$ with $v\in [0,n-1]$ and $a\in [2,n-1]$.

\smallskip

\item[(v)] $f_3^{2n-2}(af_3+f_2) \big((1-a)f_3+f_2 \big)(bf_3+f_1) \big((1-b)f_3+f_1 \big)$ with $a,b\in [2,n-1]$ and $a\ge b$.

\smallskip

\item[(vi)] $(\prod_{i=1}^{2n}(f_3+d_i)) f_2f_1$ where $S=\prod_{i=1}^{2n}d_i\in \mathcal F(\langle f_1,f_2\rangle)$ with $\sigma(S)=f_1+f_2$.
\end{enumerate}
\end{lemma}

\medskip
\begin{theorem} \label{3.9}
Let $H$ be a transfer Krull monoid over a group $G$ where $G \cong C_2\oplus C_2\oplus C_{2n}$ with $n\ge 2$. Then $\Delta_{\rho} (H)=\{1\}$.
\end{theorem}

\begin{proof}
By \eqref{eq:basic5}, we may consider $\mathcal B (G)$ instead of $H$.
Let $S$ be a minimal zero-sum sequence of length $\mathsf D(G)$ over $G$.  By Corollary \ref{3.3}.2, it suffices to prove that $1\in  \Delta \big( \supp( (-S)S ) \big)$.   We distinguish five cases induced by the structural description given by Lemma \ref{3.8}, and use Lemma \ref{3.4}.1 without further mention.

\medskip\noindent
CASE 1: \ $S = f_3^{v_3}(f_3+f_2)^{v_2}(af_3+f_1)(-af_3+f_2+f_1)$ with $a\in [1,n-1]$ as in Lemma \ref{3.8}.(i) or (ii).
\smallskip

Since
\[
W=f_3^{2n-1}(f_3+f_2)(af_3+f_1)(-af_3+f_2+f_1)\in \mathcal A \big( \supp( (-S)S ) \big)
\]
and
\[
W^2=f_3^{2n}\cdot f_3^{2n-2}(f_3+f_2)^2\cdot (af_3+f_1)^2(-af_3+f_2+f_1)^2 \,,
\]
we obtain that $1\in \Delta \big(\supp( (-S)S ) \big)$.

\medskip\noindent
CASE 2: \ $S = f_3^{2n-1}(af_3+f_2)(bf_3+f_1)(cf_3+f_2+f_1)$ as in Lemma \ref{3.8}.(iii).
\smallskip

Suppose that $c\ge n+2$. Then $S^2=f_3^{2n}\cdot f_3^{2n-2a}(af_3+f_2)^2\cdot f_3^{2a-2}(bf_3+f_1)^2(cf_3+f_2+f_1)^2$, where $f_3^{2n-2a}(af_3+f_2)^2$ and  $f_3^{2a-2}(bf_3+f_1)^2(cf_3+f_2+f_1)^2$ are atoms, and hence $1\in \Delta \big(\supp( (-S)S ) \big)$.

Suppose that $c\le n-1$. Then
\[
W_1=(-f_3)^{2a}(af_3+f_2)^2, \ W_2=(-f_3)^{2b}(bf_3+f_1)^2, \ W_3=(-f_3)^{2c}(cf_3+f_2+f_1)^2 \,,
\]
 and $W=(-f_3)(af_3+f_2)(bf_3+f_1)(cf_3+f_2+f_1)$ are atoms with $W_1W_2W_3=W^2\cdot \big((-f_3)^{2n} \big)^2$ whence $1\in \Delta \big( \supp((-S)S) \big)$.

\medskip\noindent
 CASE 3: \ $S = f_3^{2n-1-2v}(f_3+f_2)^{2v}f_2(af_3+f_1) \big((1-a)f_3+f_2+f_1 \big)$ as in Lemma \ref{3.8}.(iv).
\smallskip

Then $\{f_3, -f_3, f_2, af_3+f_1, (1-a)f_3+f_2+f_1\}\subset \supp ((-S)S)$.
Since $W=(-f_3)f_2(af_3+f_1) \big((1-a)f_3+f_2+f_1 \big)$ is an atom of length $4$, we have that $\min \Delta \big(\supp( (-S)S ) \big) \t 2$.

Setting
\[
W_1=(af_3+f_1)^2(-f_3)^{2a} \quad \text{and} \quad W_2= \big((1-a)f_3+f_1+f_2 \big)^2f_3^{2a-2}
\]
we observe that  $W_1W_2(f_2)^2=W^2(f_3(-f_3))^{2a-2}$.  Therefore $\min \Delta \big( \supp((-S)S) \big) \t 2a-3$ which implies that $\min \Delta \big(\supp((-S)S) \big)=1$.

\medskip\noindent
CASE 4: \ $S = f_3^{2n-2}(af_3+f_2) \big((1-a)f_3+f_2 \big)(bf_3+f_1) \big((1-b)f_3+f_1 \big)$ as in Lemma \ref{3.8}.(v).
\smallskip

Since $(-f_3)(af_3+f_2) \big((1-a)f_3+f_2 \big)$ is an atom of length $3$ over $\supp( (-S)S)$, we have that $1\in \Delta \big(\supp((-S)S) \big)$.

\medskip\noindent
CASE 5: \ $S = (\prod_{i=1}^{2n}(f_3+d_i)) f_2f_1$ with $T=\prod_{i=1}^{2n}d_i$ and $\sigma(T)=f_1+f_2$ as in Lemma \ref{3.8}.(vi).
\smallskip

Since $\sigma(T)\neq 0$, we have $|\supp(T)|\ge 2$, say $d_1\neq d_2$. If $d_1+d_2\in \{f_1,f_2\}$, then $(f_3+d_1)(-f_3+d_2)(d_1+d_2)$ is an atom of length $3$ over $\supp((-S)S)$ which implies that $1\in \Delta \big(\supp((-S)S) \big)$. If $d_1+d_2=f_1+f_2$, then $W_1=(f_3+d_1)(-f_3+d_2)f_1f_2$ and $W_2=(f_3+d_1)^2(-f_3+d_2)^2$ are atoms with $W_1^2=W\cdot f_1^2\cdot f_2^2$ whence $1\in \Delta \big( \supp( (-S)S ) \big)$.
\end{proof}

\medskip
\begin{lemma} \label{3.10}
Let $G$ be a finite abelian group with rank $\mathsf r(G)\ge 2$ and $\exp(G)\ge 3$, and let $U\in \mathcal A(G)$  with  $|U|=\mathsf D(G)$. If  there exist independent elements $e_1,\ldots,e_t$ with $t\ge 2$ and an element $g$ such that $\{e_1,\ldots, e_t,g\} \subset \supp(U)$ and $ag=k_1e_1+\ldots+k_te_t$ for some $a\in [1,\ord(g)-1]\setminus\{\frac{\ord(g)}{2}\}$ and with $k_i\in [1, \ord(e_i)-1]$ for all $i\in [1,t]$,  then $\min \Delta \big( \supp((-U)U) \big)=1$.  In particular, if $\supp(U)$ contains a basis of $G$, then $\min \Delta \big( \supp((-U)U) \big)=1$.
\end{lemma}

\begin{proof}
Let $(e_1,\ldots,e_t)$ be independent with $t\ge 2$ and let $g \in G $ such that $\{e_1,\ldots, e_t,g\}\subset \supp(U)$ and $ag=k_1e_1+\ldots+k_te_t$ for some $a\in [1,\ord(g)-1]\setminus\{\frac{\ord(g)}{2}\}$ and with $k_i\in [1, \ord(e_i)-1]$ for every $ i\in [1,t]$.

Now we  assume that $a \in [1,\ord(g)-1]\setminus\{\frac{\ord(g)}{2}\}$ is minimal such that $ag\in \langle e_1,\ldots, e_t\rangle$ which implies that $a\t \ord(g)$ and hence $a\in [1,\lfloor\frac{\ord(g)}{2}\rfloor-1]$. For every $i \in [1,t]$,
we replace $e_i$ by $-e_i$, if necessary, in order to  obtain $k_i\le\ord(e_i)/2 $. Thus we obtain that  $\{e_1,\ldots, e_t\}\subset \supp((-U)U)$ such that $ag=k_1e_1+\ldots+k_te_t$ with $k_i\in [1, \lfloor \ord(e_i)/2\rfloor]$ for every $i\in [1,t]$. Since $a\neq \frac{\ord(g)}{2}$, there exists $i\in [1,t]$, say $i=1$, such that $k_1\neq \ord(e_1)/2$. Now we distinguish two cases.

\smallskip\noindent
CASE 1: \ For all $i \in [1,t]$, we have  $k_i\neq \ord(e_i)/2$.

Then,  by the minimality of $a$,
\[
W_1=g^ae_1^{\ord(e_1)-k_1}e_2^{\ord(e_2)-k_2}\prod_{i\in [3,t]}(-e_i)^{k_i} \quad \text{ and} \quad  W_2=g^{2a}e_1^{\ord(e_1)-2k_1}e_2^{\ord(e_2)-2k_2}\prod_{i\in [3,t]}(-e_i)^{2k_i}
\]
are atoms over $\supp( (-U)U)$. Since $W_1^2=W_2\cdot e_1^{\ord(e_1)}\cdot e_2^{\ord(e_2)}$, we infer that $1\in \Delta \big(\supp((-U)U) \big)$ which implies that $\min \Delta \big(\supp((-U)U) \big) = 1$.

\smallskip\noindent
CASE 2: \ There exists $i\in [2,t]$ such that $k_i=\ord(e_i)/2$.

After renumbering if necessary, there exists $t_0\in [1,t-1]$ such that $k_i\neq \ord(e_i)/2$ for every $i\in [1,t_0]$ and $k_i=\ord(e_i)/2$ for every $i\in [t_0+1,t]$. Then  \[
V_1=g^a\prod_{i\in [1,t]}(-e_i)^{k_i} \quad \text{ and } \quad V_2=g^ae_1^{\ord(e_1)-k_1}\prod_{i\in [2,t]}(-e_i)^{k_i}
\]
are atoms over $\supp((-U)U)$. Since
\[
\begin{aligned}
V_1^2 & = g^{2a}\prod_{i\in [1,t_0]}(-e_i)^{2k_i}\cdot \prod_{i\in [t_0+1,t]}(-e_i)^{\ord(e_i)} \,, \\
V_2^2 & = g^{2a}e_1^{\ord(e_1)-2k_1}\prod_{i\in [2,t_0]}(-e_i)^{2k_i}\cdot \prod_{i\in [t_0+1,t]}(-e_i)^{\ord(e_i)}\cdot e_1^{\ord(e_1)} \,,
\end{aligned}
\]
and $g^{2a}\prod_{i\in [1,t_0]}(-e_i)^{2k_i},\  g^{2a}e_1^{\ord(e_1)-2k_1}\prod_{i\in [2,t_0]}(-e_i)^{2k_i}$ are atoms,
we infer that $$\min \Delta \big( \supp((-U)U) \big)\t \gcd(1+t-t_0-2, 1+t-t_0+1-2 )$$ whence $\min \Delta \big(\supp((-U)U) \big)=1$.

\smallskip
To show the in particular part,  let $\{e_1, \ldots, e_t\}\subset \supp(U)$ be a basis of $G$, and note that  $t\ge \mathsf r(G)$ by \cite[Lemma A.6]{Ge-HK06a}.  For each $i\in [1,t]$, we set $I_i=\{g\in \supp(U) \mid g\in \langle e_i\rangle\}$ and $T_i=\prod_{g\in I_i}g^{\mathsf v_g(U)}$. Then
\[
U=T_1 \cdot \ldots \cdot T_tT , \quad \text{where} \quad 1 \ne T=\prod_{g\in \supp(U)\setminus\cup_{i\in[1,t]}I_i}g^{\mathsf v_g(U)} \,.
\]
 Therefore for every $g\in \supp(T)$, there exists a subset $J \subset [1,t]$ with $|J|\ge 2$ such that  $g=\sum_{j\in J}k_je_j$, where $k_j\in [1, \ord(e_j)-1]$ for each $j\in J$. If  $\ord(g)\neq 2$ for some $g\in \supp(T)$, then the assumptions of the main case hold whence $\min \Delta \big(\supp((-U)U) \big)=1$.

Now suppose that $\ord(g)=2$ for each $g\in \supp(T)$.  Then $\sigma(T_1) \cdot \ldots  \cdot \sigma(T_t)\sigma(T)$ is an atom, $\ord(\sigma(T))=2$,  and $\sigma(T_i)\in \langle e_i\rangle$ for each $i\in [1,t]$. It follows that $\sigma(T_i)=\frac{\ord(e_i)}{2}e_i$ for each $i\in [1,t]$, $|T|=1$, and $\sigma (T)=\frac{\ord(e_1)}{2}e_1+\ldots +\frac{\ord(e_t)}{2}e_t$. Since $|U|=\mathsf D(G)\ge \mathsf D^*(G)\ge 1+\sum_{j=1}^t(\ord(e_j)-1)$ by \cite[Proposition 5.1.7]{Ge-HK06a}, we have $|T_j|=\ord(e_j)-1$ for each $i\in [1,t]$.
Since $\exp(G)\ge 3$, we may assume that $\ord(e_1)\ge 3$ after renumbering if necessary.
Since $e_1\in \supp(T_1)$ and $T_1$ is a zero-sum free sequence over $\langle e_1\rangle$ of length $\ord(e_1)-1$,  we obtain $\sigma(T_1)=-e_1=\frac{\ord(e_1)}{2}e_1$ by \cite[Theorem 5.1.10]{Ge-HK06a}, a contradiction to $\ord(e_1)\ge 3$.
\end{proof}

\medskip
\begin{theorem} \label{3.11}
Let $H$ be a transfer Krull monoid over a group $G$ where $G=C_{p^k}^r$ with $k,r \in \N$, $r \ge 2$, and  $p \in \P$ such that $p^k\ge 3$. Then $\Delta_{\rho} (H) = \{1\}$.
\end{theorem}

\begin{proof}
By \eqref{eq:basic5}, it is sufficient to consider $\mathcal B (G)$ instead of $H$.
 By Corollary \ref{3.3}.2, we only need to show that $\min \Delta \big(\supp((-U)U) \big) = 1$ for every atom $U\in \mathcal A(G)$ of length $\mathsf D(G)$.
Let $U$ be an atom of length $\mathsf D(G)$. Then $\langle\supp(U)\rangle=G$ by \cite[Proposition 5.1.4]{Ge-HK06a}, and hence $\supp(U)$ contains a basis of $G$ by \cite[Lemma A.7]{Ge-HK06a}. Now Lemma \ref{3.10} implies that $\min \Delta \big(\supp((-U)U) \big) = 1$.
\end{proof}

\smallskip
If $G$ is an elementary $2$-group of rank $r\ge 3$, then the above assumption of Lemma \ref{3.10} never holds true. Thus elementary $2$-groups need a different approach.

\medskip
\begin{lemma} \label{3.12}
Let $G$ be an elementary $2$-group of rank $r\ge 3$ and let $U,V\in \mathcal A(G)$ be distinct atoms of length $\mathsf D(G)$. Then $1\in \Delta(\mathsf L(UV^2))$.
\end{lemma}
\begin{proof}
Since $U$ and $V$ are distinct, there exists an element $g\in \supp(U)\setminus \supp(V)$, and clearly $\supp(U)\setminus \{g\}$ is a basis of $G$. We set $\supp(U)\setminus \{g\}=\{e_1,\ldots, e_r\}$,  $g=e_0=e_1+\ldots+e_r$, and then $U= e_0e_1 \cdot \ldots \cdot e_r$. Since $\{e_1, \ldots, e_r\}$ is a basis of $G$, $V$ can be written in the form $V=e_{I_1} \cdot \ldots \cdot e_{I_{r+1}}$, where  $\emptyset \ne I_j \subset [1,r]$ and  $e_{I_j}=\sum_{i\in I_j}e_i$ for every  $j\in [1,r+1]$. We continue with the following assertion.

\smallskip
\begin{enumerate}
\item[{\bf A.}\,] There exist two distinct $k_1,k_2 \in [1,r+1]$  such that   $I_{k_1}\cap I_{k_2}\neq \emptyset$, $I_{k_1}\setminus I_{k_2}\neq \emptyset$, and $I_{k_2}\setminus I_{k_1}\neq \emptyset$.
\end{enumerate}

\noindent
{\it Proof of} \,{\bf A}.\, First, we choose $I$, say $I=I_1$, to be maximal in $\{I_j\mid j\in[1,r+1]\}$. Note that $e_0\not\in \supp(V)$ and hence $I_j\neq [1,r]$ for every $j\in[1,r+1]$. Since $I_1\subset \cup_{j\in[2,r+1]}I_j$, we can choose $K\subset [2,r+1]$ to be minimal such that $I_1\subset \cup_{j\in K}I_j$. Then $I\cap I_k\neq \emptyset$ and $I\setminus I_k\neq \emptyset$ for all $k\in K$. If there exists $k\in K$ such that $I_k\setminus I_1\neq\emptyset $, then we are done. Otherwise, $I_k\subset I_1$ for all $k\in K$. By the maximality of $I_1$, we know that $|K|\ge 2$ and by the minimality of $K$, we have  that $I_{k_1}\setminus I_{k_2}\neq \emptyset$ and $I_{k_2}\setminus I_{k_1}\neq \emptyset$ for every two distinct  $k_1$ and $k_2$. Assume to the contrary that $I_{k_1}\cap I_{k_2}= \emptyset$  for every distinct  $k_1$ and $k_2$. Thus $e_{I_1}\prod_{k\in K}e_{I_k}$ is an atom, a contradiction to $|V|=\mathsf D(G)$. \qed[Proof of {\bf A}]

\smallskip
After renumbering if necessary, we suppose that $I_{1}\cap I_{2}\neq \emptyset$, $I_{1}\setminus I_{2}\neq \emptyset$, and $I_{2}\setminus I_{1}\neq \emptyset$. We define
\[
W_1=e_{I_1}e_{I_2}\prod_{i\in (I_1\cup I_2)\setminus (I_1\cap I_2)}e_i \,, \quad  \quad W_2=e_0e_{I_1}e_{I_2}\prod_{i\not\in (I_1\cup I_2)\setminus (I_1\cap I_2) }e_i \,,
\]
and observe that $W_1,W_2$ are atoms. Since
\[
UV^2 = U\cdot e_{I_1}^2\cdot e_{I_2}^2\cdot \prod_{j\in [3,r+1]}e_{I_j}^2
= W_1\cdot W_2\cdot \prod_{j\in [3,r+1]}e_{I_j}^2\,,
\]
we obtain that $1\in\Delta (\mathsf L(UV^2))$.
\end{proof}

\bigskip
\begin{theorem} \label{3.13}
Let $H$ be a transfer Krull  monoid over an elementary $2$-group $G$  of rank $r\ge 2$. Then $\Delta_{\rho}^*(H)=\Delta_{\rho}(H)=\{1, r-1\}$.
\end{theorem}

\begin{proof}
By \eqref{eq:basic5}, it is sufficient to consider $\mathcal B (G)$ instead of $H$.
Let $(e_1,\ldots, e_r)$ be a basis of $G$ and $S= e_0e_1 \cdot \ldots \cdot e_r \in \mathcal A(G)$, where $e_0=e_1+\ldots +e_r$. Then $\Delta \big(\supp(S) \big)=\{r-1\}$ and hence $r-1\in \Delta_{\rho}^*(G)$. By Theorem \ref{3.5}, we have that $\Delta_{\rho}(G)\supset \Delta_{\rho}^*(G)\supset\{1, r-1\}$. Thus it remains to prove that $\Delta_{\rho}(G)\subset \{1,r-1\}$.

Since $\max \Delta_{\rho}(G) \le \max \Delta (G) =r-1$ by \cite[Theorem 6.7.1]{Ge-HK06a}, we may suppose that  $r\ge 4$. Assume to the contrary that there exists $d\in \Delta_{\rho}(G)\setminus \{1,r-1\}$.
Then for every $k \in \N$ there is a $B_k \in \mathcal B (G)$ such that $\rho \big( \mathsf L (B_k) \big) = \mathsf D (G)/2$ and $\mathsf L (B_k)$ is an AAP with difference $d$ and length $\ell \ge k$.  Lemma \ref{3.2}.1 implies that $B_k$ is a product of atoms having length $\mathsf D(G)$.  We fix $k=|\{A\in\mathcal A(G)\mid |A|=\mathsf D(G) \}|+1$. If $B_k=U^t$ with $t\in \N$  for some  $U\in \mathcal A(G)$ with $|U|=\mathsf D(G)$, then $r-1=\min \Delta \big(\supp(U) \big)=\min \Delta \big( \supp(B_k) \big)\t d$, a contradiction. Otherwise, the choice of $k$ implies that there are distinct atoms $U,V\in \mathcal A(G)$ with $|U|=|V|=\mathsf D(G)$ such that $U^2V\t B_k$. By Lemma \ref{3.12}, $1\in \Delta(\mathsf L(U^2V))\subset \Delta(\mathsf L(B_k))$ and hence $d\t 1$, a contradiction.
\end{proof}

\bigskip
\begin{theorem} \label{3.14}
Let $H$ be a transfer Krull monoid over a finite cyclic group $G$  of order $n \ge 3$. Then $n-2\in \Delta_{\rho}^*(H)=\Delta_{\rho}(H)$.
\end{theorem}

\begin{proof}
By \eqref{eq:basic5}, it is sufficient to consider $\mathcal B (G)$ instead of $H$.
Since $ n-2 \in \Delta_{\rho}^*(G) \subset \Delta_{\rho}(G)$,  it remains to verify that  $\Delta_{\rho}(G)\subset \Delta_{\rho}^*(G)$.

Let $d\in \Delta_{\rho}(G)$. Then for every $k \in \N$ there is a $B_k \in \mathcal B (G)$ such that $\rho \big( \mathsf L (B_k) \big) = \mathsf D (G)/2$ and $\mathsf L (B_k)$ is an AAP with difference $d$ and length $\ell \ge k$. Thus $\gcd \Delta(\mathsf L(B_k))=d$.
We set $k=n(n-1)+1$,  $G_0=\supp(B_k)$, and claim that $\min \Delta(G_0)=\gcd \Delta(\mathsf L(B_k))$ which implies that $d=\min \Delta(G_0)\in \Delta_{\rho}^*(G)$.

Clearly,  $\min \Delta(G_0)\t d$, and hence it remains to prove that $d\t \min \Delta(G_0)$.
By Lemma \ref{3.2}, $B_k$ is a product of atoms having length $\mathsf D (G)=n$. Note that $|\supp(U)|=1$ for all atoms of length $n$ and $|\{U\in \mathcal A(G)\mid |U|=n\}|\le n-1$.
 Thus $k=n(n-1)+1$ implies that  $B_k$ is a product of the form
\[
B_k = U_1^{n+1}U_2\cdot  \ldots \cdot U_r\,,
\]
where $r\in \N$, $U_1, \ldots, U_r$ are  atoms of length $n$, and  $U_1=g^n$, where $g\in G$ with $\ord(g)=n$.

Then for every atom $V\in \mathcal A(G_0)$,  we have $V \t U_1\cdot  \ldots \cdot U_r$ and $\{n+1, \| V\|_g+n\} \subset \mathsf L (U_1^nV)$. Therefore $d\t \| V\|_g-1$ for all $V\in \mathcal A(G_0)$ whence $d$ divides $\gcd \{\| V\|_g-1\mid V\in \mathcal A(G_0) \}$. Since $\min \Delta(G_0)=\gcd\{\| V\|_g-1\mid V\in \mathcal A(G_0) \}$ by Lemma \ref{3.4}.3, the claim follows.
\end{proof}

\medskip
\begin{corollary} \label{3.15}
We have $\Delta_{\rho}(C_4)=\{2\}$, $\Delta_{\rho}(C_5)=\{1,3\}$,  $\Delta_{\rho}(C_6)=\{4\}$, $\Delta_{\rho}(C_7)=\{1,5\}$, $\Delta_{\rho}(C_8)=\{1,6\}$, $\Delta_{\rho}(C_9)=\{1,7\}$,
$\Delta_{\rho}(C_{10})=\{2, 8\}$,  $\Delta_{\rho}(C_{11})=\{1,9\}$, $\Delta_{\rho}(C_{12})=\{1,10\}$.
\end{corollary}

\begin{proof}
Let $G$ be a cyclic group of order $|G| =n \in [4,12]$. By  Theorem \ref{3.14}, we infer that $n-2 \in \Delta_{\rho}^* (G)=\Delta_{\rho} (G)$. By Theorem \ref{3.5}, we have $1 \in \Delta_{\rho} (G)$ if and only if $n \notin \{4,6,10\}$. Lemma \ref{3.2} shows that
\[
\Delta_{\rho}^* (G) = \{\min \Delta (G_0) \mid G_0=-G_0 \ \text{and} \ \ord (g)=n \ \text{for every} \ g \in G_0 \} \,.
\]
Now we use Lemma \ref{3.4}.3.
If $n \in \{4,6\}$, then for some $g \in G$ with $\ord (g)=n$ we get
\[
\Delta_{\rho}^* (G) = \{ \min \Delta (\{g,-g\}) = \{n-2\} \,.
\]
If $n=10$, then for some $g \in G$ with $\ord (g)=n$ we get
\[
\begin{aligned}
\Delta_{\rho}^* (G) & = \{ \min \Delta (\{g,-g\}), \min \Delta (\{3g,-3g\}), \min \Delta (\{g,-g, 3g, -3g\}) \} \\
 & = \{2,8\} \,.
\end{aligned}
\]
Suppose that $n \in [4,12] \setminus \{4,6,10]$. Let $G_0 \subset G$ be a subset consisting of elements of order $n$ and with $G_0 =-G_0$. If $|G_0|=2$, then $\min \Delta (G_0)=n-2$. Suppose that $|G_0|>2$. Then there is some $g \in G_0$ and some $k \in \N$ with $\gcd (k,n)=1$ such that $\{g,-g, kg, -kg\} \subset G_0$. Then $\min \Delta (G_0)$ divides $\min \Delta ( \{g,-g, kg, -kg\})$ and, by going through all cases and using Lemma \ref{3.4}.3,  we obtain that $\min \Delta ( \{g,-g, kg, -kg\}) =1$.  Thus the assertion follows.
\end{proof}

In the next lemma we need some basics from the theory of continued fractions (see \cite{HK13a} for some background; in particular, we use Theorems 2.1.3 and 2.1.7 of \cite{HK13a}).

\medskip
\begin{lemma}\label{3.16}
Let $G$ be a cyclic group with order  $n>3$, $g \in G$ with $\ord (g)=n$,  and $a \in [2,n-1]$ with $\gcd(a,n)=1$. Let $[a_0,\ldots, a_m]$ be the continued fraction expansion of $n/a$ with odd length (i.e. $m$ is even).
\begin{enumerate}
\item $\min \Delta(\{g, ag\})=\gcd(a_1, a_3,\ldots, a_{m-1})<n-2$ and $\min \Delta (\{g,-g,  ag, -ag\})\in \Delta_{\rho}^*(G)$.

\smallskip
\item If $a<n/2$, then $\min \Delta(\{g, ag, -ag, -g\})=\gcd( a_0-1, a_1, \ldots, a_{m-1}, a_m-1)$. Note that this also holds for the continued fraction expansion of $n/a$ with even length and hence this holds for the regular continued fraction expansion of $n/a$ (i.e. $a_m>1$).
\end{enumerate}
\end{lemma}

\begin{proof}
1. For the first part, see \cite[Theorem 2.1]{Ch-Ch-Sm07b} or \cite[Theorem 1]{Ge90d}. For the second part, since $g^n$ and $(ag)^n$ are two atoms of length  $\mathsf D(G)$, we obtain $\rho \big(\mathsf L(g^n(-g)^n(ag)^n(-ag)^n) \big)=\mathsf D(G)/2$ which implies
$\min \Delta(\{g,-g,  ag, -ag\})\in \Delta_{\rho}^*(G)$ by Lemma \ref{3.2}.3.

2. Suppose that $a < n/2$. By Lemma \ref{3.4}.3, we have
\begin{align*}
&\min \Delta(\{g, ag, -ag, -g\})\\
=& \gcd\{\| V\|_g-1\mid V\in \mathcal A(\{g, ag, -ag, -g\}) \} \\
=& \gcd\{\| V\|_g-1\mid V\in \mathcal A(\{g, ag\})\cup \mathcal A(\{g, -ag\})\cup \mathcal A(\{-g, ag\})\cup \mathcal A(\{-g, -ag\})  \}\\
=& \gcd\{\| V\|_g-1\mid V\in \mathcal A(\{g, ag\})\cup \mathcal A(\{g, -ag\})  \}\\
=& \gcd\{\min \Delta(\{g, ag\}), \min \Delta(\{g, -ag\})\}\,.
\end{align*}
Since the continued fraction of $\frac{n}{n-a}$ with odd length is
\begin{equation*}
\left\{
\begin{aligned}
&[1,a_0-1, a_1,\ldots, a_m-1, 1]   \text{ if $a_m>1$}, \\
&[1,a_0-1, a_1,\ldots, a_{m-1}+1]  \text{ if $a_m=1$},
\end{aligned}
\right.
\end{equation*}
1. implies that $\min \Delta(\{g, ag\})=\gcd(a_1, a_3,\ldots, a_{m-1})$ and \begin{equation*}
\min \Delta(\{g, -ag\})=\left\{
\begin{aligned}
&\gcd(a_0-1, a_2, a_4, \ldots, a_m-1)   \text{ if $a_m>1$}, \\
&\gcd(a_0-1, a_2, a_4, \ldots, a_{m-2})    \text{ if $a_m=1$}.
\end{aligned}
\right.
\end{equation*}
Therefore, we obtain
\[
\min \Delta(\{g, ag, -ag, -g\})=\gcd(\min \Delta(\{g, -ag\}), \min \Delta(\{g, ag\}) )=\gcd( a_0-1, a_1, \ldots, a_{m-1}, a_m-1) \,.
\]
\end{proof}

\medskip
\begin{theorem} \label{3.17}
Let $H$ be a transfer Krull monoid over a finite cyclic group $G$  of order $n \ge 3$.
Then  the following statements are equivalent{\rm \,:}
\begin{enumerate}
\item[(a)] $\Delta_{\rho}^* (H)\setminus \{1, n-2\}\neq \emptyset$.

\item[(b)] There is an $a  \in [2, \lfloor n/2\rfloor]$ with  $\gcd (n, a)=1$ such that $\gcd(a_0-1, a_1, \ldots, a_{m-1}, a_m-1)>1$, where $[a_0,a_1,\ldots,a_m]$ is the regular continued fraction expansion of $n/a$ (i.e. $a_m>1)$.
\end{enumerate}
\end{theorem}

\begin{proof}
By \eqref{eq:basic5}, it is sufficient to prove the equivalence for $\mathcal B (G)$ instead of $H$.

\smallskip
(a)\, $\Rightarrow$ \, (b) \ Note that for any distinct atoms $U,V$ of length $n$, we have $\min \Delta \big(\supp( (-U)U(-V)V) \big)<n-2$ by Lemma \ref{3.16}.1.  Since $\Delta_{\rho}^* (H)\setminus \{1, n-2\}\neq \emptyset$, there must exist distinct atoms $U,V$ of length $n$ such that $\min \Delta \big(\supp((-U)U(-V)V) \big)\in \Delta_{\rho}^*(G)\setminus \{1,n-2\}$. Let $U=g^n$ and $V=(ag)^n$, where $g\in G$ and $a\in [2,n-2]$ with $\gcd(n,a)=1$. Then let $G_0=\{g, ag, -g, -ag\}$. If $a\ge \frac{n}{2}$,  then  $n-a\le \frac{n}{2}$. Thus we assume that $a\le \frac{n}{2}$. Therefore Lemma \ref{3.16}.2 implies that $\gcd(a_0-1, a_1, \ldots, a_{m-1}, a_m-1)>1$, where $[a_0,a_1,\ldots,a_m]$ is the regular continued fraction expansion of $n/a$.

(b)\, $\Rightarrow$ \, (a) \ We set $G_0=\{g,ag,-g,-ag\}$ where $g \in G$ with $\ord (g)=n$. Then $\min \Delta(G_0)<n-2$ and Lemma \ref{3.16}.2 implies that $\min \Delta(G_0)>1$. It follows that $\Delta_{\rho}^* (H)\setminus \{1, n-2\}\neq \emptyset$.
\end{proof}

\medskip
\begin{corollary} \label{3.18}
Let $G$ be a cyclic group of order  $n> 4$, and let $g \in G$ with $\ord (g)=n$.
\begin{enumerate}
\item If $n$ is even and $n-1$ is not a prime, then there is an even $d\in \Delta_{\rho}^*(G)\setminus\{1, n-2\}$.

\item If $n$ is even, $3\not|\  n$, and $n-3$ is not a prime, then there is an even $d\in \Delta_{\rho}^*(G)\setminus\{1, n-2\}$.

\item If $n$ is even and $n\equiv 2q\pmod {q^2}$ for some odd prime $q$ with $q^2+2q\le n$, then there is an even $d\in \Delta_{\rho}^*(G)\setminus\{1, n-2\}$.

\item If $n$ is even and $n\equiv q\pmod {2q+1}$ for some odd $q$ with $5q+2\le n$, then there is an even $d\in \Delta_{\rho}^*(G)\setminus\{1, n-2\}$.

\item If $n$ is even with $n \in [8, 10^9]$, then  $\Delta_{\rho}^*(G)=\{1, n-2\}$ if and only if \\
$n\in \{8,12,14,18,20,30,32,44,48,54,62,72,74,84,90,102,138,182,230,252,270,450,462,2844\}$.

\item If $n>5$ is odd and $n-1$ is a square, then there is an odd $d\in \Delta_{\rho}^*(G)\setminus\{1, n-2\}$.
\end{enumerate}

\end{corollary}

\begin{proof}
Note that if $a \in [2, n-1]$ with $\gcd (a,n)=1$, then $\min \Delta ( \{g, ag, -g, -ag\} ) \in \Delta^*_{\rho} (G)$ and $\min \Delta ( \{g, ag, -g, -ag\} )<n-2$  by Lemma \ref{3.16}.1.

1. Let $n=mt+1$ be even with $m\in [2,n-2]$, and set $G_0=\{g, mg, -mg, -g\}$. Then $m,t$ are odd, $\gcd(m,n)=1$, and $m<n/2$. Since $[t,m]$ is the regular continued fraction of $n/m$, we have that $\min \Delta(G_0)=\gcd(m-1, t-1)$ is even  and hence $\min \Delta(G_0)\in \Delta_{\rho}^*(G)\setminus\{1, n-2\}$.

2. If $n\equiv 1\pmod 3$, then $n-1$ is not a prime and hence 1. implies the assertion. Suppose $n\equiv 2\pmod 3$ and let $n-3=m_1m_2$ with $1<m_1<n-3$. Then there exists $i\in [1,2]$, say $i=1$, such that $m_1\equiv 1\pmod 3$. Set $G_0=\{g, m_1g, -m_1g, -g\}$. Since $n$ is even, we obtain that $m_1,m_2$ are odd and hence $\lfloor \frac{m_1}{3}\rfloor$ is even.   Since $[m_2, \lfloor \frac{m_1}{3}\rfloor, 3]$ is the regular continued fraction of $n/m$, we have that $\min \Delta(G_0)=\gcd(m_2-1, \lfloor \frac{m_1}{3}\rfloor, 2)=2$  by Lemma \ref{3.16}.1 and hence $\min \Delta(G_0)\in \Delta_{\rho}^*(G)\setminus\{1, n-2\}$.

3. Let $n=q^2t+2q$ be even with $m=qt+1$, and set $G_0=\{g, mg, -mg, -g\}$. Then $n=qm+q$ and $t\ge 1$ is even.
Since $[q,t,q]$ is the regular continued fraction of $n/m$, we have that $\min \Delta(G_0)=\gcd (q-1, t, q-1)$ is even by Lemma \ref{3.16}.1 and hence $\min \Delta(G_0)\in \Delta_{\rho}^*(G)\setminus\{1, n-2\}$.

4. Let $n=(2q+1)t+q$ be even with $t$  odd, and set $G_0=\{g, (2q+1)g, -(2q+1)g, -g\}$. Then $\gcd(2q+1,n)=1$ and $5q+2\le n$ implies that $2q+1<n/2$. Since $[t,2,q]$ is the regular continued fraction of $n/(2q+1)$, we have that $\min \Delta(G_0)=\gcd (t-1, 2, q-1)=2$  by Lemma \ref{3.16}.1 and hence $\min \Delta(G_0)\in \Delta_{\rho}^*(G)\setminus\{1, n-2\}$.

5. This was done by a computer program.

6. Let  $n=m^2+1$ be odd, and set $G_0=\{g, mg, -mg, -g\}$. Then $m$ is even.   Since $[m, m]$ is the regular continued fraction of $n/m$, we have that $\min \Delta(G_0)=\gcd(m-1, m-1)=m-1>1$ is odd by Lemma \ref{3.16}.1 and hence $\min \Delta(G_0)\in \Delta_{\rho}^*(G)\setminus\{1, n-2\}$.
 \end{proof}

\smallskip
Next we discuss an application of Theorem \ref{3.17} to the so-called
Characterization Problem which is in the center of all arithmetical investigations of transfer Krull monoids. It asks whether two finite abelian groups $G$ with $\mathsf D (G) \ge 4$ and $G'$, whose systems of sets of lengths $\mathcal L (G)$ and $\mathcal L (G')$ coincide, have to be isomorphic (for an overview on this topic we refer \cite[Section 6]{Ge16c}). It is well-known that for every $n \ge 4$, the systems $\mathcal L (C_n)$ and $\mathcal L (C_2^{n-1})$ are distinct and that $\mathcal L (C_2^{n-1}) \not\subset \mathcal L (C_n)$ (\cite[Theorem 3.5]{Ge-Sc-Zh17b}).
If $n \in [4,5]$, then $\mathcal L (C_n) \subset \mathcal L (C_2^{n-1})$ (\cite[Section 4]{Ge-Sc-Zh17b}), but for $n \ge 6$ there is no information available so far. The results of the present section yield the following corollary.

\medskip
\begin{corollary} \label{3.19}
Let $G$ be a cyclic group of order $n \ge 6$. If the equivalent statements in Theorem \ref{3.17} hold, then $\mathcal L (C_n) \not\subset \mathcal L (C_2^{n-1})$.
\end{corollary}

\noindent
{\it Comment.} Note that Corollary \ref{3.18} shows that the equivalent statements in Theorem \ref{3.17} hold true for infinitely many $n \in \N$.

\begin{proof}
Assume to the contrary that $\mathcal L (C_n) \subset \mathcal L (C_2^{n-1})$. Then $\Delta_{\rho} (C_n) \subset \Delta_{\rho} (C_2^{n-1})$. Since $\Delta_{\rho} (C_2^{n-1}) = \{1, n-2\}$ by Theorem \ref{3.13}, we obtain a contradiction to Theorem \ref{3.17}.
\end{proof}

\smallskip
We end this section with the following conjecture (note, if $G$ is cyclic of order three or isomorphic to $C_2 \oplus C_2$, then $\Delta_{\rho} (G)=\{1\}$).

\medskip
\begin{conjecture} \label{3.20}
Let $H$ be a transfer Krull monoid over a finite abelian group $G$ with $|G| > 4$. Then $\Delta_{\rho} (H) = \{1\}$ if and only if $G$ is neither cyclic nor an elementary $2$-group.
\end{conjecture}

We summarize what follows so far by the results of the present section. Clearly, one implication of Conjecture \ref{3.20} holds true. Indeed,
if $G$ is cyclic or an elementary $2$-group with $|G|>4$, then $\Delta_{\rho} (H) \ne \{1\}$ by Theorems \ref{3.13} and \ref{3.14}. Conversely, for groups of rank two, and for groups isomorphic either to $C_2 \oplus C_2 \oplus C_{2n}$ or to $C_{p^k}^r$, where $n, r  \ge 2$, $k \ge 1$, and $p$ is a  prime with $p^k\ge 3$,  the conjecture holds true by Theorems \ref{3.7}, \ref{3.9}, and \ref{3.11} (consequently,  the conjecture holds true for all groups $G$ with $|G| \in [5, 47]$). In view of our discussion (preceding Lemma \ref{3.2}) on the state of the art on the Davenport constant, Conjecture \ref{3.20} might seem to be quite bold, but it is consistent with all what we know on the Davenport constant so far. Indeed, let $U \in \mathcal A (G)$ with $|U|= \mathsf D (G)$. The goal is to show that $\min \Delta \big( \supp ( (-U)U) \big) = 1$. By \cite[Proposition 5.1.11]{Ge-HK06a}, $\supp (U)$ contains a generating set of $G$. If it contains a basis, then we are done by Lemma \ref{3.10}.
Suppose $G$ is as in \eqref{eq:basic6} with $\mathsf D (G) = \mathsf D^* (G)$, $\mathsf r (G)=r >1$, and $(e_1, \ldots, e_r)$ is a basis with $\ord (e_i)=n_i$ for all $i \in [1,r]$. Then
\[
U = e_1^{n_1-1} \cdot \ldots \cdot e_r^{n_r-1}(e_1+ \ldots + e_r)
\]
is the canonical example of a minimal zero-sum sequence of length $\mathsf D^* (G)$. Clearly, there are minimal zero-sum sequences of different form (as Lemma \ref{3.6} shows for $r=2$) but their support can only be greater than or equal to $\mathsf r (G)+1$  (recall that $\mathsf r (G) = \min \{|G_0| \mid G_0 \subset G \ \text{is a generating set} \}$ by \cite[Lemma A.6]{Ge-HK06a}). Furthermore,  for subsets  $G_0 \subset G_1$  of $G$, we have $\min \Delta (G_1) \le \min \Delta (G_0)$.  The combination of these two facts provides strong support for the above conjecture.

\medskip
\section{Weakly Krull monoids} \label{4}
\medskip

The main goal in this section is to study the set $\Delta_{\rho}(\cdot)$ for $v$-noetherian weakly Krull monoids and for their monoids of $v$-invertible $v$-ideals. Our main result is given by Theorem \ref{4.4}.

We start with the local case, namely with finitely primary monoids.
A monoid $H$ is said to be {\it finitely primary} if there are  $s, \alpha \in \N$ and a factorial monoid $F = F^{\times} \time \mathcal F ( \{p_1, \ldots, p_s\})$ such that $H \subset F$ with
\begin{equation} \label{eq:basic2}
H \setminus H^{\times} \subset p_1 \cdot \ldots \cdot p_s F \quad \text{and} \quad (p_1 \cdot \ldots \cdot p_s)^{\alpha}F \subset H \,.
\end{equation}
In this case $s$ is called the rank of $H$ and $\alpha$ is called an exponent of $H$. It is well-known (\cite[Theorems 2.9.2 and 3.1.5]{Ge-HK06a}) that $F$ is the complete integral closure of $H$, that
\begin{equation} \label{eq:basic3}
\text{$H$ has finite elasticity if and only if $s=1$} \,,
\end{equation}
and that
\begin{equation} \label{eq:basic4}
\text{$H/H^{\times}$ is finitely generated if and only if $s=1$ and $(F^{\times} \DP H^{\times}) < \infty$} \,.
\end{equation}

To provide some examples of finitely primary monoids, we first recall that
every numerical monoid $H \subsetneq (\N_0,+)$ is  finitely generated and finitely primary of rank one with accepted elasticity $\rho (H)>1$. Furthermore, if $R$ is a one-dimensional local Mori domain, $\widehat R$ its complete integral closure, and   $(R \DP \widehat R) \ne \{0\}$, then its multiplicative monoid of non-zero elements is finitely primary (\cite[Sections 2.9, 2.10,  and 3.1]{Ge-HK06a}). Note that a  finitely primary monoid $H$ with  $\rho (H) > 1$ is not a transfer Krull monoid by \cite[Theorem 5.5]{Ge-Sc-Zh17b}.

Our first lemma is known for numerical monoids (\cite[Theorem 2.1]{Ch-Ho-Mo06} and \cite[Proposition 2.9]{B-C-K-R06}).

\smallskip
\begin{lemma} \label{4.1}
Let $H \subset F=F^{\times}\times \mathcal F ( \{p\})$ be a finitely primary monoid of rank $1$ and exponent $\alpha$, and let $\mathsf v = \mathsf v_p \colon H \to \N_0$ denote the homomorphism onto the value semigroup of $H$. Suppose that $\{\mathsf v (a) \mid a \in \mathcal A (H) \} = \{n_1, \ldots, n_s\}$ with $1 \le n_1 < \ldots < n_s$.
Then $\mathsf v (H) \subset \N_0$ is a numerical monoid, and we have
\begin{enumerate}
\item  $\rho (H) = n_s/n_1$,  and if $F^{\times}/H^{\times}$ is a torsion group, then the elasticity is accepted.

\smallskip
\item  Let $d = \gcd  \{n_i - n_{i-1} \mid i \in [2,s]\}$. Then $d \mid \gcd  \Delta (H) $ and if $|F^{\times}/H^{\times}|=1$, then $d = \gcd  \Delta (H) $.
\end{enumerate}
\end{lemma}

\begin{proof}
If $a \in \mathcal A (H)$, then $p^{\alpha} F \subset H$ (see  \eqref{eq:basic2}) implies $\mathsf v (a) \le 2\alpha -1$, and hence $n_s \le 2 \alpha - 1$. Since $\N_{\ge \alpha} \subset \mathsf v (H)$, it follows that  $\mathsf v (H) \subset \N_0$ is a numerical monoid.

\smallskip
1. To show that $\rho (H) \le n_s/n_1$, let $a \in H$ be given and suppose that $a=u_1 \cdot \ldots \cdot u_k= v_1 \cdot \ldots \cdot v_{\ell}$ where $k,\ell \in \N$ and $u_1, \ldots , u_k,v_1, \ldots, v_{\ell} \in \mathcal A (H)$. Then
\[
\ell n_1 \le \sum_{i=1}^{\ell} \mathsf v (v_i) = \mathsf v (a) = \sum_{i=1}^k \mathsf v (u_i) \le k n_s \,,
\]
whence $\ell /k \le n_s/n_1$ and thus $\rho ( \mathsf L (a)) \le n_s/n_1$.

To show that $\rho (H) = n_s/n_1$, let $u_1=\epsilon_1p^{n_1}, u_2=\epsilon_2 p^{n_s} \in \mathcal A (H)$ with $\epsilon_1, \epsilon_2 \in F^{\times}$, and let $s \in \N_0$ such that $sn_1n_s \ge \alpha$. Then for every $k > s$ we have
\[
u_2^{k n_1} = \epsilon_2^{kn_1}p^{kn_1n_s} = \big( \epsilon_2^{kn_1} \epsilon_1^{-(k-s)n_s} p^{sn_1n_s}\big)  \big(\epsilon_1p^{n_1}\big)^{(k-s)n_s} = \big( \epsilon_2^{kn_1} \epsilon_1^{-(k-s)n_s} p^{sn_1n_s}\big) u_1^{(k-s)n_s} \,.
\]
Thus
\[
\rho (\mathsf L (u_2^{kn_1})) = \frac{\max \mathsf L (u_2^{kn_1})}{\min \mathsf L (u_2^{kn_1})} \ge \frac{1+(k-s)n_s}{kn_1}
\]
tends to $n_s/n_1$ as $k$ tends to infinity.

Now suppose that  $F^{\times}/H^{\times}$ is a torsion group, and let $u_1, u_2$ be as above. Then there is a $k_0 \in \N$ such that $(\epsilon_2^{n_1}\epsilon_1^{-n_s})^{k_0} \in H^{\times}$. Then the above calculation with $k=k_0$ and $s=0$ shows that $\rho (\mathsf L (u_2^{k_0n_1})) = n_s/n_1$.

\smallskip
2. For every $i \in [1,s]$ there are $t_i \in \N_0$ such that $n_i=n_1+t_id$. Since $p^{\alpha}F \subset H$, it follows that $\gcd (n_1, d)=1$.  Let $a \in H$ and consider two factorizations
\[
a = \prod_{i=1}^s \prod_{j=1}^{k_i} u_{i,j} = \prod_{i=1}^s \prod_{j=1}^{\ell_i} v_{i,j} \,,
\]
where all $u_{i,j}, v_{i,j}$ are (not necessarily distinct) atoms with $\mathsf v (u_{i,j}) = n_i=\mathsf v (v_{i,j})$ for all $i \in [1,s]$. Then
\[
\mathsf v (a) = \sum_{i=1}^s k_in_i = \sum_{i=1}^s \ell_i n_i = \sum_{i=1}^s \ell_i (n_1+t_id)
\]
whence
\[
n_1 \, \sum_{i=1}^s (\ell_i-k_i) = d \, \sum_{i=1}^s (k_i-\ell_i)t_i
\]
and this implies that $d$ divides $\sum_{i=1}^s (\ell_i-k_i)$. Thus $d$ divides $\gcd \Delta (H)=\min \Delta (H)$.

Now suppose that $F^{\times}=H^{\times}$. We show that $\gcd \Delta (H)$ divides $n_i - n_{i-1}$ for every $i \in [2,s]$ which implies that $\gcd \Delta (H)$ divides $d$ and equality follows. Let $i \in [2,s]$. Then there are atoms $u_{i-1}=\epsilon_{i-1}p^{n_{i-1}}$ and $u_i = \epsilon_i p^{n_i}$ with $\epsilon_{i-1}, \epsilon_i \in F^{\times}=H^{\times}$. Then
\[
u_i^{n_{i-1}} = \big( \epsilon_i p^{n_i} \big)^{n_{i-1}} = \big( \epsilon_{i-1}p^{n_{i-1}} \big)^{n_i} (\epsilon_i^{n_{i-1}} \epsilon_{i-1}^{-n_i} ) = u_{i-1}^{n_i} \eta \,,
\]
where $\eta = \epsilon_i^{n_{i-1}} \epsilon_{i-1}^{-n_i} \in H^{\times}$. Thus $\gcd \Delta (H)$ divides $n_i - n_{i-1}$.
\end{proof}

We continue with  simple examples showing that the elasticity need not be accepted if $F^{\times}/H^{\times}$ fails to be a torsion group, and that $d$ need not be equal to $\min \Delta (H)$.

\smallskip
\begin{example} \label{4.2}~

1. Let $H\subset F$ be a finitely primary monoid as in \eqref{4.1}, and  generated by $\{\epsilon_1p^2, \epsilon_2p^4, \epsilon p^3 \mid \epsilon \in F^{\times} \}$, where $\epsilon_1, \epsilon_2 \in F^{\times}$ with  $\ord(\epsilon_1) = \infty$ and $\ord(\epsilon_2) < \infty$. We assert that $\rho(H)$ is not accepted.

First, we observe that  $\mathcal A(H)=\{\epsilon_1p^2, \epsilon_2p^4, \epsilon p^3 \mid \epsilon \in F^{\times} \}$. Thus Lemma \ref{4.1}.1 implies that $\rho (H)=2$.  For every $b\in H$, we have $\mathsf v ( b ) \le 4\min\mathsf L(b)$ and $\mathsf v (b) \ge 2\max\mathsf L(b)$ which infer that $\rho( \mathsf L(b)) \le 2$. Assume to the contrary that  $\rho( \mathsf L (b))=2$. Then $\mathsf v ( b) =4\min\mathsf L(b)=2\max\mathsf L(b)$ which implies that
$b=(\epsilon_2p^4)^{\min \mathsf L(b)}=(\epsilon_1p^2)^{\max \mathsf L(b)}$. It follows that $\epsilon_2^{\min \mathsf L(b)}=\epsilon_1^{2\min \mathsf L(b)}$, a contradiction to our assumption on $\ord(\epsilon_1)$ and $\ord(\epsilon_2)$. Therefore $\rho( \mathsf L(b))<2$ for all $b\in H$ whence $\rho (H)$ is not accepted.

\smallskip
2. Let $F^{\times} = \{\epsilon\}$ with $\epsilon^2=1$, and $H = \langle \epsilon p^3, p^5 \rangle \subset F=F^{\times} \time \mathcal F (\{p\})$. Then $\min \Delta (H)= 4 > 2 = d$, where  $d$ as in Lemma \ref{4.1}.2.
\end{example}

\medskip
\begin{lemma} \label{4.3}~

\begin{enumerate}
\item Let $H$ be a finitely primary monoid with accepted elasticity $\rho (H)>1$. Then $\Delta_{\rho}^* (H) = \Delta_{\rho} (H) = \Delta_1 (H) = \{\min \Delta (H)\}$.

\smallskip
\item Let $H=H_1 \times \ldots \times H_n$ where $n \in \N$ and $H_i$ is a  finitely primary  monoid with accepted elasticity and $\min \Delta (H_i)=d_i$ for all $i \in [1,n]$.
Suppose that   $\rho (H_1) =  \ldots = \rho (H_s)=\rho (H) > \rho (H_i)$   for all $i \in [s+1, n]$. Then $\min \Delta_{\rho} (H) =\min \Delta_{\rho}^* (H) = \gcd (d_1, \ldots, d_s)$, $\max \Delta_{\rho} (H)=\max \Delta_{\rho}^* (H)$, and
\[
\big\{ \, \gcd  \{d_i \mid i \in I\}  \ \mid  \emptyset \ne I \subset [1,s] \big\} =\Delta_{\rho}^* (H) \  \subset \ \Delta_{\rho} (H) \subset \big\{ d \in \N \mid d  \ \text{divides some} \ d' \in \Delta_{\rho}^* (H) \big\} \,.
\]
\end{enumerate}
\end{lemma}

\begin{proof}
1. By Lemmas \ref{2.2} and \ref{2.4}, we have
\[
\{ \min \Delta ( \LK a \RK ) \mid a \in H \ \text{with} \ \rho (\mathsf L (a))=\rho (H) \} = \Delta_{\rho}^* (H) \subset \Delta_{\rho} (H) \subset \Delta_1 (H) \,.
\]
If $a \in H$ with $\rho (\mathsf L (a))=\rho (H) > 1$, then $a \in H \setminus H^{\times}$ and hence $\LK a \RK = H$. Thus it remains to show that $\Delta_1 (H) = \{\min \Delta (H) \}$, which follows from \cite[Theorem 4.3.6]{Ge-HK06a}.

\smallskip
2. Without restriction we may suppose that $H$ is reduced. Then also $H_1, \ldots, H_n$ are reduced.   We use Lemma \ref{2.6}. Note that $H_1, \ldots, H_n$ need not be finitely generated whence Lemma \ref{2.4}.3 cannot be applied to the present setting.

Let $a = a_1 \cdot \ldots \cdot a_n \in H$ with $a_i \in H_i$ for all $i \in [1,n]$. If $\rho ( \mathsf L (a)) = \rho (H)$, then $a_{s+1}=\ldots=a_n=1$ and
\[
\LK a \RK = \prod_{i \in [1,s], a_i \ne 1} H_i \,.
\]
For every $i \in [1,s]$, 1. implies that $\Delta_{\rho} (H_i)=\{d_i\}$. If $\emptyset \ne I \subset [1,s]$, then \cite[Proposition 1.4.5]{Ge-HK06a} implies that
\[
\gcd  \Delta \big( \prod_{i\in I} H_i \big)  = \gcd \, \bigcup_{i \in I} \Delta (H_i)  \,,
\]
and clearly
\[
\gcd \, \bigcup_{i \in I} \Delta (H_i)  = \gcd \, \{ \gcd \Delta (H_i) \mid i \in I \}  = \gcd  \{ d_i \mid i \in I \} \,.
\]
Thus we obtain that (the first equality follows from Lemma \ref{2.2}.2)
\[
\begin{aligned}
\Delta_{\rho}^* (H) & =  \Big\{ \gcd \Delta ( \LK a \RK ) \mid a \in H \ \text{with} \ \rho (\mathsf L (a))=\rho (H) \Big\} \\
 & = \Big\{\gcd \Delta (\prod_{i \in I} H_i) \mid \emptyset \ne I \subset [1,s] \Big\} \\
 & = \Big\{ \gcd  \{d_i \mid i \in I \} \mid \emptyset \ne I \subset [1,s] \Big\}  \,.
\end{aligned}
\]
Since $\Delta_{\rho} (H) = \Delta_{\rho} (H_1 \time \ldots \time H_s)$,  $\min \Delta (H_1 \time \ldots \time H_s) = \gcd (d_1, \ldots, d_s)$, and $\min \Delta_{\rho}^* (H) = \gcd (d_1, \ldots, d_s)$, it follows that $\min \Delta_{\rho} (H) =  \gcd (d_1, \ldots, d_s) $.

Lemma \ref{2.4}.1 implies that $\Delta_{\rho}^* (H)   \subset  \Delta_{\rho} (H)$, and it remains to show that $\Delta_{\rho} (H) \subset \big\{ d \in \N \mid d  \ \text{divides some} \ d' \in \Delta_{\rho}^* (H) \big\}$. If this holds, then we immediately get that $\max \Delta_{\rho} (H)=\max \Delta_{\rho}^* (H)$.
Now let $d \in \Delta_{\rho} (H)$ be given. We claim that $d$ divides some element from $\Delta_{\rho}^* (H)$.

For every $k \in \N$ there is some $a^{(k)} \in H$ such that $\mathsf L (a^{(k)})$ is an AAP with difference $d$, length at least $k$, and with $\rho \big( \mathsf L (a^{(k)}) \big) = \rho (H)$. Let $k \in \N$. Then $a^{(k)} = a_1^{(k)} \cdot \ldots \cdot a_s^{(k)}$ with $a_i^{(k)} \in H_i$ and $\rho \big( \mathsf L (a_i^{(k)}) \big) = \rho (H_i)=\rho (H)$ for all $i \in [1,s]$. Then there is a subsequence $b^{({\ell})} = a^{(k_{\ell})}$ of $a^{(k)}$, a nonempty subset $I \subset [1,s]$, say $I=[1,r]$, and a constant $M$ such that the following holds for every $k \in \N$.
\begin{itemize}
\item For every $i \in [1,r]$, $\mathsf L (b_i^{(k)})$ is an AAP with difference $d_i$, length at least $k$, and with $\rho \big( \mathsf L (B_i^{(k)}) \big) = \rho (H)$.

\item For every $i \in [r+1,s]$, $|\mathsf L (b_i^{(k)})| \le M$.
\end{itemize}
Thus $\mathsf L ( b_1^{(k)} \cdot \ldots \cdot b_r^{(k)}) = \mathsf L ( b_1^{(k)}) + \ldots + \mathsf L ( b_r^{(k)})$ is an AAP with difference $\gcd (d_1, \ldots, d_r) \in \Delta_{\rho}^* (H)$ and length growing with $k$. Since $\mathsf L (b^{(k)})$ is an AAP with difference $d$, it follows that $d$ divides $\gcd (d_1, \ldots, d_r)$.
\end{proof}

\medskip
For our discussion of weakly Krull monoids we put together some notation and gather their main properties. For any undefined notion we refer to \cite{HK98, Ge-HK06a}. In the remainder of this sections all monoids are commutative and cancellative and by a domain we always mean a commutative integral domain. If $R$ is a domain, then its semigroup $R^{\bullet} = R \setminus \{0\}$ of non-zero elements is a monoid.

Let $H$ be a monoid. Then $\mathsf q (H)$ denotes its quotient group,
\[
\widehat H = \{ x \in \mathsf q (H) \mid \text{there is a $c \in H$ such that} \ cx^n \in H \ \text{for all} \ n \in \N \} \subset \mathsf q (H) \,,
\]
its complete integral closure, and $(H \DP \widehat H) = \{ x \in \mathsf q (H) \mid x \widehat H \subset H\}$ the conductor of $H$. Furthermore, $H_{\red} = \{ a H^{\times} \mid a \in H \}$ is the associated reduced monoid of $H$ and $\mathfrak X (H)$ is the set of minimal non-empty prime $s$-ideals of $H$. Let $\mathcal I_v^* (H)$ denote the monoid of $v$-invertible $v$-ideals of $H$ (together with $v$-multiplication). Then $\mathcal F_v(H)^{\times} = \mathsf q \big( \mathcal I_v^* (H) \big)$ is the quotient group of fractional $v$-invertible $v$-ideals, and $\mathcal C_v (H) = \mathcal F_v(H)^{\times}/\{ xH \mid x \in \mathsf q (H)\}$ is the $v$-class group of $H$.

The monoid $H$ is said to be  {\it weakly Krull} (\cite[Corollary 22.5]{HK98}) if
\[
H = \bigcap_{\mathfrak p \in \mathfrak X (H)} H_{\mathfrak p} \quad \text{and} \quad \{\mathfrak p \in \mathfrak X (H) \mid a \in \mathfrak p \} \ \text{is finite for all} \ a \in H \,.
\]
If $H$ is $v$-noetherian, then $H$ is weakly Krull if and only if  $v$-$\max (H) = \mathfrak X (H)$ (\cite[Theorem 24.5]{HK98}). A domain $R$ is {\it weakly Krull} if $R^{\bullet}$ is a weakly Krull monoid. Weakly Krull domains were introduced by Anderson, Anderson, Mott, and Zafrullah (\cite{An-An-Za92b, An-Mo-Za92}), and weakly Krull monoids by Halter-Koch (\cite{HK95a}).
The monoid $H$ is Krull if and only if $H$ is weakly Krull and $H_{\mathfrak p}$ is a discrete valuation monoid for each $\mathfrak p \in \mathfrak X (H)$.

Every saturated submonoid $H$ of a monoid $D = \mathcal F (P) \times D_1 \ldots \times D_n$, where $P$ is a set of primes and $D_1, \ldots, D_n$ are primary monoids, is weakly Krull if the class group $\mathsf q (D)/\big( D^{\times}\mathsf q (H) \big)$ is a torsion group (\cite[Lemma 5.2]{Ge-Ka-Re15a}).
We mention a few key examples examples of $v$-noetherian weakly Krull monoids and domains and refer to \cite[Examples 5.7]{Ge-Ka-Re15a} for a detailed discussion. Suppose that $H$ is as in Theorem \ref{4.4}. Then, by the previous remark,  its monoid of $v$-invertible $v$-ideals $\mathcal I_v^* (H)$ is a weakly Krull monoid. Furthermore, all one-dimensional noetherian domains are $v$-noetherian weakly Krull.  If  $R$ is $v$-noetherian weakly Krull domain with non-zero conductor $(R \DP \widehat R)$ and $\mathfrak p \in \mathfrak X (R)$, then $R_{\mathfrak p}^{\bullet}$ is finitely primary, and thus the assumption made in Theorem \ref{4.4} holds.
Orders in algebraic number fields are one-dimensional noetherian and hence they are $v$-noetherian weakly Krull domains. If $R$ is an order, then its $v$-class group $\mathcal C_v (R)$ (which coincides with the Picard group) as well as the index of the unit groups $(\widehat R^{\times} \DP R^{\times})$ are finite and every class contains a minimal prime ideal $\mathfrak p \in \mathcal P$. Thus all assumptions made in Theorem \ref{4.4}.4 are satisfied. It was first proved by Halter-Koch (\cite[Corollary 4]{HK95b}) that the elasticity of orders in number fields is accepted whenever it is finite.

\medskip
\begin{theorem} \label{4.4}
Let $H$ be a $v$-noetherian weakly Krull monoid with  conductor
 $\emptyset \ne \mathfrak f = (H \DP \widehat H) \subsetneq H$ such that $H_{\mathfrak p}$ is finitely primary for each $\mathfrak p \in \mathfrak X (H)$. Let $\mathcal P^* = \{ \mathfrak p \in \mathfrak X(H) \mid \mathfrak p \supset \mathfrak f \}$, $\mathcal P = \mathfrak X (H) \setminus \mathcal P^*$, and let $\pi \colon \mathfrak X (\widehat H) \to \mathfrak X (H)$ be the natural map defined by $\pi ( \mathfrak P) = \mathfrak P \cap H$ for all $\mathfrak P \in \mathfrak X (\widehat H)$.
\begin{enumerate}
\item $\mathcal I_v^* (H)$ has finite elasticity if and only if $\pi$ is bijective.

\smallskip
\item  If $\pi$ is bijective and   $\widehat H_{\mathfrak p}^{\times} / H_{\mathfrak p}^{\times}$ are torsion groups for all $\mathfrak p \in \mathcal P^*$, then $\mathcal I_v^* (H)$ has accepted elasticity.

\smallskip
\item Suppose that  $\mathcal I_v^* (H)$ has accepted elasticity, and let $\mathfrak p_1, \ldots, \mathfrak p_s \in \mathcal P^*$ be the minimal prime ideals with $\rho \big(H_{\mathfrak p_i} \big) = \rho \big( \mathcal I_v^* (H) \big)$ for all $i \in [1,s]$, and set $d_i = \min \Delta (H_{\mathfrak p_i})$. Then
    \[
    \begin{aligned}
\big\{  \gcd  \{d_i \mid i \in I\}  \ \mid  \emptyset \ne I \subset [1,s] \big\} & =\Delta_{\rho}^* \big( \mathcal I_v^* (H) \big) \  \subset \ \Delta_{\rho} \big( \mathcal I_v^* (H) \big) \\ & \subset \big\{ d \in \N \mid d  \ \text{divides some} \ d' \in \Delta_{\rho}^* \big( \mathcal I_v^* (H) \big) \big\} \,.
    \end{aligned}
    \]

\smallskip
\item Let  $G_{\mathcal P} \subset \mathcal C_v (H)$ denote the set of classes containing a minimal prime ideal from $\mathcal P$. Suppose that $\pi$ is bijective, and that  $\mathcal C_v (H)$ and $\widehat H^{\times}/H^{\times}$ are both finite. Then $H$ has accepted elasticity and if $\rho (H)= \rho (G_{\mathcal P})$, then $\Delta_{\rho} (  G_{\mathcal P} ) \subset \Delta_{\rho} (H)$.
\end{enumerate}
\end{theorem}

\begin{proof}
By \cite[Section 5]{Ge-Ka-Re15a}), we infer that  $\widehat H$ is Krull, $\mathcal P^*$ is finite, and that
\begin{equation} \label{structure}
\mathcal I_v^* (H)   \ \ito \ \mathcal F( \mathcal P ) \time T \,, \quad \text{where} \quad T = \prod_{\mathfrak p \in \mathcal P^*} (H_{\mathfrak p})_{\red} \,.
\end{equation}

\smallskip
1. This follows from \eqref{eq:basic3},  from \eqref{structure}, and from Lemma \ref{2.6}.1.

\smallskip
2. This follows from Lemma \ref{2.6}.1 and from Lemma \ref{4.1}.1.

\smallskip
3. This follows from \eqref{structure} and from Lemma \ref{4.3}.2.

\smallskip
4.  There is a transfer homomorphism
\[
\boldsymbol \beta \colon H \to \mathcal B (H) \,, \quad \text{where} \quad \mathcal B (H) \hookrightarrow \mathcal F (G_{\mathcal P}) \times T
\]
is the $T$-block monoid of $H$ and  the inclusion is saturated and cofinal (\cite[Definition 3.4.9]{Ge-HK06a}). Thus $\mathcal L \big( \mathcal B (H) \big) = \mathcal L (H)$, whence it suffices to prove all the statements for $\mathcal B (H)$ instead of proving them for $H$.

Since $\mathcal C_v (H)$ and $\widehat H^{\times}/H^{\times}$ are finite, the exact sequence (\cite[Proposition 5.4]{Ge-Ka-Re15a})
\[
1 \to
\widehat{H}^\times/H^\times  \to \coprod_{\mathfrak{p}\in
\mathfrak{X}(H)}\widehat{H}_{\mathfrak{p}}^\times/H^\times_{\mathfrak{p}}  \to
\mathcal{C}_v(H) \to \mathcal{C}_v(\widehat{H}) \to 0   \,,
\]
implies that $(\widehat H_{\mathfrak p}^{\times} \DP H_{\mathfrak p}^{\times}) < \infty$ for all $\mathfrak p \in \mathcal P^*$. Thus, by \ref{eq:basic4}, all factors of $T$ are finitely generated and hence $T$ is finitely generated. Therefore $\mathcal B (H)$ is finitely generated (as a saturated submonoid of a finitely generated monoid) and hence $\mathcal B (H)$ has accepted elasticity by \cite[Theorem 3.1.4]{Ge-HK06a}.

Since $\mathcal B (G_{\mathcal P}) \subset \mathcal B (H)$ is a divisor-closed submonoid, the remaining statement follows from Lemma \ref{2.4}.2.
\end{proof}

\medskip
\begin{remarks} \label{4.5}~

1. Let $H$ be as in Theorem \ref{4.4}. If $\pi$ is bijective and $H$ is seminormal, then $\mathcal I_v^* (H)$ is half-factorial (\cite[Theorem 5.8.1.(a)]{Ge-Ka-Re15a}) and hence $\Delta \big( \mathcal I_v^* (H) \big) = \emptyset$.

\smallskip
2. Let $R$ be a noetherian weakly Krull domain such that its integral closure $\overline R$ is a finitely generated $R$-module. Then, for $\mathfrak p \in \mathcal P^*$,  the index $(\overline R_{\mathfrak p}^{\times} \DP R_{\mathfrak p}^{\times})$ is finite if and only if $R/\mathfrak p$ is finite (\cite[Theorem 2.1]{Ka98}).

\smallskip
3. Lemma \ref{4.1} shows that the elasticity of a finitely primary monoid of rank $1$ is completely determined by its value semigroup. The interplay of algebraic and arithmetical properties of one-dimensional local Mori domains with properties of their value semigroup has found wide attention in the literature (\cite{Ba-Do-Fo97, Ba-An-Fr00a,DA16a}).

\smallskip
4. For every $d \in \N$, there is a $v$-noetherian finitely primary monoid $H$ with $\min \Delta (H)=d$. However, even for orders $R$ in algebraic number fields the precise value of $\min \Delta (R_{\mathfrak p})$, $\mathfrak p \in \mathcal P^*$, is known only  for some explicit examples (as discussed in \cite[Examples 3.7.3]{Ge-HK06a}).

To consider the global case, let $H$ is as in Theorem \ref{4.4} with   finite $v$-class group $\mathcal C_v (H)$, and suppose further that every class contains a minimal prime ideal from $\mathcal P$.  If $H$ is seminormal or  $|G|\ge 3$, then $\min \Delta (H)=1$ (\cite[Theorem 1.1]{Ge-Zh16c}).
\end{remarks}

\smallskip
It is a central but far open problem in factorization theory to characterize when a weakly Krull monoid $H$ and when its monoid $\mathcal I_v^* (H)$ of $v$-invertible $v$-ideals are transfer Krull monoids resp. transfer Krull monoids of finite type.
To begin with the local case, finitely primary monoids are not transfer Krull  and the same is true for finite direct products of finitely primary monoids (\cite[Theorem 5.6]{Ge-Sc-Zh17b}). These are one of the spare results available so far which indicate that weakly Krull monoids (with the properties of Theorem \ref{4.4}) are transfer Krull only in exceptional cases.  Clearly, combining results from Section \ref{3} with Theorem \ref{4.4}.3 we obtain examples of when the system of sets of lengths of $\mathcal I_v^* (H)$ does not coincide with $\mathcal L (G)$ for any resp. some finite abelian groups $G$. Clearly, if $\mathcal L \big( \mathcal I_v^* (H) \big) \ne \mathcal L (G)$ for an abelian group $G$, then  $\mathcal I_v^* (H)$ is not transfer Krull over $G$.

We formulate one such result (others would be possible) as a corollary. But, of course,  we are far away from a characterization of when $H$ and the monoid $\mathcal I_v^* (H)$ are transfer Krull resp. of when $\mathcal L (H)$ or $\mathcal L ( \mathcal I_v^* (H))$ coincide with $\mathcal L (G)$ for some finite abelian group $G$ (see Section 5 and Problem 5.9 in \cite{Ge-Sc-Zh17b}).

\smallskip
\begin{corollary} \label{4.6}
Let $H$ be a $v$-noetherian weakly Krull monoid with  conductor
 $\emptyset \ne \mathfrak f = (H \DP \widehat H) \subsetneq H$ such that $H_{\mathfrak p}$ is finitely primary for each $\mathfrak p \in \mathfrak X (H)$ and $\mathcal I_v^* (H)$ has accepted elasticity.
Let $\mathfrak p_1, \ldots, \mathfrak p_s$ be the minimal prime ideals with $\rho \big(H_{\mathfrak p_i} \big) = \rho \big( \mathcal I_v^* (H) \big) > 1$.
\begin{enumerate}
 \item  If $\gcd \big(\min \Delta (H_{\mathfrak p_1}), \ldots, \min \Delta (H_{\mathfrak p_s}) \big) > 1$ and $G$ is  a finite abelian group with  $\mathcal L \big(\mathcal I_v^* (H) \big) = \mathcal L (G)$, then $G$ is cyclic of order $4$, $6$, or $10$.

 \item If there is an $i \in [1,s]$ with $\min \Delta (H_{\mathfrak p_i}) > 1$ and $G$ is a finite abelian group with $\mathcal L \big(\mathcal I_v^* (H) \big) = \mathcal L (G)$, then $G$ does not have rank two and is not of the form $C_{p^k}^r$ with $k,r \in \N$, $r\ge 2$, and $p$ prime with $p^k \ge 3$. Moreover, if Conjecture \ref{3.20} holds true, then $G$ is either cyclic or isomorphic to $C_2^{1+ \min \Delta (H_{\mathfrak p_i})}$.
\end{enumerate}
\end{corollary}

\begin{proof}
1.  We set $d = \gcd \big(\min \Delta (H_{\mathfrak p_1}), \ldots, \min \Delta (H_{\mathfrak p_s}) \big)$. Then Theorem \ref{4.4}.3 and Lemma \ref{4.3}.2 imply that $\min \Delta_{\rho} \big(\mathcal I_v^* (H) \big) = d$. Thus the assertion follows from Theorem \ref{3.5}.

\smallskip
2. We set $\mathfrak p = \mathfrak p_i$, $\min \Delta (H_p) = d$, and let $G$ be a finite abelian group such that $\mathcal L (G) = \mathcal L \big(\mathcal I_v^* (H) \big)$. Then Theorem \ref{4.4}.3 implies that $d \in \Delta_{\rho}^* \big( \mathcal I_v^* (H) \big) \subset \Delta_{\rho} \big( \mathcal I_v^* (H) \big) = \Delta_{\rho}(G)$. Thus the assertion follows from Theorems \ref{3.7},  \ref{3.11}, \ref{3.13} and Conjecture \ref{3.20}.
\end{proof}

\bigskip
\noindent
{\bf Acknowledgement.} We thank the referees for their careful reading.

\providecommand{\bysame}{\leavevmode\hbox to3em{\hrulefill}\thinspace}
\providecommand{\MR}{\relax\ifhmode\unskip\space\fi MR }
\providecommand{\MRhref}[2]{%
  \href{http://www.ams.org/mathscinet-getitem?mr=#1}{#2}
}
\providecommand{\href}[2]{#2}


\begin{thebibliography}{10}

\bibitem{An-An-Za92b}
D.D. Anderson, D.F. Anderson, and M.~Zafrullah, \emph{Atomic domains in which
  almost all atoms are prime}, Commun. Algebra \textbf{20} (1992), 1447 --
  1462.

\bibitem{An-Mo-Za92}
D.D. Anderson, J.~Mott, and M.~Zafrullah, \emph{Finite character
  representations for integral domains}, Boll. Unione Mat. Ital. \textbf{6}
  (1992), 613 -- 630.

\bibitem{Ba-Sm15}
N.R. Baeth and D.~Smertnig, \emph{Factorization theory: {F}rom commutative to
  noncommutative settings}, J. Algebra \textbf{441} (2015), 475 –-- 551.

\bibitem{Ba-An-Fr00a}
V.~Barucci, M.~D'Anna, and R.~Fr{\"o}berg, \emph{Analytically unramified
  one-dimensional semilocal rings and their value semigroups}, J. Pure Appl.
  Algebra \textbf{147} (2000), 215 -- 254.

\bibitem{Ba-Do-Fo97}
V.~Barucci, D.E. Dobbs, and M.~Fontana, \emph{Maximality {P}roperties in
  {N}umerical {S}emigroups and {A}pplications to {O}ne-{D}imensional
  {A}nalytically {I}rreducible {L}ocal {D}omains}, vol. 125, Memoirs of the
  Amer. Math. Soc., 1997.

\bibitem{B-C-K-R06}
C.~Bowles, S.T. Chapman, N.~Kaplan, and D.~Reiser, \emph{On delta sets of
  numerical monoids}, J. Algebra Appl. \textbf{5} (2006), 695 -- 718.

\bibitem{Ch-Ch-Sm07b}
S.~Chang, S.T. Chapman, and W.W. Smith, \emph{On minimum delta set values in
  block monoids over cyclic groups}, Ramanujan J. \textbf{14} (2007), 155 --
  171.

\bibitem{C-F-G-O16}
S.T. Chapman, M.~Fontana, A.~Geroldinger, and B.~Olberding (eds.),
  \emph{Multiplicative {I}deal {T}heory and {F}actorization {T}heory},
  Proceedings in Mathematics and Statistics, vol. 170, Springer, 2016.

\bibitem{Ch-Ho-Mo06}
S.T. Chapman, M.~Holden, and T.~Moore, \emph{Full elasticity in atomic monoids
  and integral domains}, Rocky Mt. J. Math. \textbf{36} (2006), 1437 -- 1455.

\bibitem{DA16a}
M.~D'Anna, \emph{Ring and semigroup constructions}, Multiplicative ideal theory
  and factorization theory, Springer Proc. Math. Stat., vol. 170, Springer,
  [Cham], 2016, pp.~97--115.

\bibitem{Fa-Tr18a}
Y.~Fan and S.~Tringali, \emph{Power monoids: {A} bridge between factorization
  theory and arithmetic combinatorics}, {https://arxiv.org/abs/1701.09152}.

\bibitem{Fa-Zh16a}
Y.~Fan and Q.~Zhong, \emph{Products of $k$ atoms in {K}rull monoids}, Monatsh.
  Math. \textbf{181} (2016), 779 -- 795.

\bibitem{Fo-Ho-Lu13a}
M.~Fontana, E.~Houston, and T.~Lucas, \emph{Factoring {I}deals in {I}ntegral
  {D}omains}, Lecture Notes of the Unione Matematica Italiana, vol.~14,
  Springer, 2013.

\bibitem{Ge90d}
A.~Geroldinger, \emph{On non-unique factorizations into irreducible elements.
  {II}}, Number {T}heory, {V}ol. {II} {B}udapest 1987, Colloquia {M}athematica
  {S}ocietatis {J}anos {B}olyai, vol.~51, North Holland, 1990, pp.~723 -- 757.

\bibitem{Ge16c}
\bysame, \emph{Sets of lengths}, Amer. Math. Monthly \textbf{123} (2016), 960
  -- 988.

\bibitem{Ge-Gr-Yu15}
A.~Geroldinger, D.J. Grynkiewicz, and P.~Yuan, \emph{On products of $k$ atoms
  {II}}, Mosc. J. Comb. Number Theory \textbf{5} (2015), 73 -- 129.

\bibitem{Ge-HK06a}
A.~Geroldinger and F.~Halter-Koch, \emph{Non-{U}nique {F}actorizations.
  {A}lgebraic, {C}ombinatorial and {A}nalytic {T}heory}, Pure and Applied
  Mathematics, vol. 278, Chapman \& Hall/CRC, 2006.

\bibitem{Ge-Ka10a}
A.~Geroldinger and F.~Kainrath, \emph{On the arithmetic of tame monoids with
  applications to {K}rull monoids and {M}ori domains}, J. Pure Appl. Algebra
  \textbf{214} (2010), 2199 -- 2218.

\bibitem{Ge-Ka-Re15a}
A.~Geroldinger, F.~Kainrath, and A.~Reinhart, \emph{Arithmetic of seminormal
  weakly {K}rull monoids and domains}, J. Algebra \textbf{444} (2015), 201 --
  245.

\bibitem{Ge-Sc16a}
A.~Geroldinger and W.~A. Schmid, \emph{A characterization of class groups via
  sets of lengths}, {http://arxiv.org/abs/1503.04679}.

\bibitem{Ge-Sc-Zh17b}
A.~Geroldinger, W.~A. Schmid, and Q.~Zhong, \emph{Systems of sets of lengths:
  transfer {K}rull monoids versus weakly {K}rull monoids}, in {\it Rings,
  Polynomials, and Modules}, Springer, 2017,
  p.~{http://arxiv.org/abs/1606.05063}.

\bibitem{Ge-Yu12b}
A.~Geroldinger and P.~Yuan, \emph{The set of distances in {K}rull monoids},
  Bull. Lond. Math. Soc. \textbf{44} (2012), 1203 –-- 1208.

\bibitem{Ge-Zh16c}
A.~Geroldinger and Q.~Zhong, \emph{The set of distances in seminormal weakly
  {K}rull monoids}, J. Pure Appl. Algebra \textbf{220} (2016), 3713 -- 3732.

\bibitem{Ge-Zh16a}
\bysame, \emph{The set of minimal distances in {K}rull monoids}, Acta Arith.
  \textbf{173} (2016), 97 -- 120.

\bibitem{Ge-Zh17b}
\bysame, \emph{A characterization of class groups via sets of lengths {II}}, J.
  Th{\'e}or. Nombres Bordx. \textbf{29} (2017), 327 -- 346.

\bibitem{HK95a}
F.~Halter-Koch, \emph{Divisor theories with primary elements and weakly {K}rull
  domains}, Boll. Un. Mat. Ital. B \textbf{9} (1995), 417 -- 441.

\bibitem{HK95b}
\bysame, \emph{Elasticity of factorizations in atomic monoids and integral
  domains}, J. Th{\'e}or. Nombres Bordx. \textbf{7} (1995), 367 -- 385.

\bibitem{HK98}
\bysame, \emph{Ideal {S}ystems. {A}n {I}ntroduction to {M}ultiplicative {I}deal
  {T}heory}, Marcel Dekker, 1998.

\bibitem{HK13a}
\bysame, \emph{Quadratic {I}rrationals}, Pure and Applied Mathematics, vol.
  306, Chapman \& Hall/CRC, 2013.

\bibitem{Ka98}
F.~Kainrath, \emph{A note on quotients formed by unit groups of semilocal
  rings}, Houston J. Math. \textbf{24} (1998), 613 -- 618.

\bibitem{Ka16b}
\bysame, \emph{Arithmetic of {M}ori domains and monoids{\rm \,:} {T}he {G}lobal
  {C}ase}, Multiplicative {I}deal {T}heory and {F}actorization {T}heory,
  Springer Proc. Math. Stat., vol. 170, Springer, 2016, pp.~183 -- 218.

\bibitem{Pl-Sc18a}
A.~Plagne and W.A. Schmid, \emph{On congruence half-factorial {K}rull monoids
  with cyclic class group}, submitted.

\bibitem{Sc11b}
W.A. Schmid, \emph{The inverse problem associated to the {D}avenport constant
  for ${C}_2 \oplus {C}_2 \oplus {C}_{2n} $, and applications to the
  arithmetical characterization of class groups}, Electron. J. Comb.
  \textbf{18(1)} (2011), Research Paper 33.

\bibitem{Sc16a}
\bysame, \emph{Some recent results and open problems on sets of lengths of
  {K}rull monoids with finite class group}, Multiplicative {I}deal {T}heory and
  {F}actorization {T}heory, Springer Proc. Math. Stat., vol. 170, Springer,
  2016, pp.~323 -- 352.

\bibitem{Sm17a}
D.~Smertnig, \emph{{Factorizations in bounded hereditary noetherian prime
  rings}}, Proceedings Edinburgh Math. Soc., to appear {http://arxiv.org/abs/1605.09274}.

\bibitem{Sm13a}
\bysame, \emph{Sets of lengths in maximal orders in central simple algebras},
  J. Algebra \textbf{390} (2013), 1 -- 43.

\bibitem{Zh17a}
Q.~Zhong, \emph{Sets of minimal distances and characterizations of class groups
  of {K}rull monoids}, Ramanujan J., to appear.

\end{thebibliography}
\end{document}